\documentclass[letter,12pt]{article}

\usepackage{amssymb}
\usepackage{amsthm}
\usepackage{amsfonts}
\usepackage{amsmath}
\usepackage{anysize}
\usepackage{hyperref}
\usepackage{bbm}
\usepackage{color}

\marginsize{2.6cm}{2.6cm}{1cm}{2cm}

\newtheorem*{theorem*}{Theorem}

\newtheorem{theorem}{Theorem}[section]

\newtheorem{lemma}[theorem]{Lemma}
\newtheorem{corollary}[theorem]{Corollary}

\newtheorem{definition}[theorem]{Definition}
\newtheorem{proposition}[theorem]{Proposition}

\def\K{K\"ahler }
\def\i{\sqrt{-1}}
\def\del{\partial}
\def\dbar{\bar\partial}
\def\ddbar{\del\dbar}

\def\del{\partial}
\newcommand{\calH}{\mathcal{H}}

\DeclareMathOperator{\Ric}{Ric}

\def\o{\omega}

\def\h#1{\hbox{#1}}

\def\Herm{\operatorname{Herm}}
\def\calHo{\mathcal H_\o}
\def\Ho{\calHo}
\def\Pn{\mathcal P_n}
\def\P{\mathcal P}
\def\Hermn{\Herm_n}
\def\tr{\operatorname{tr}}
\def\qq{\qquad}

\def\FS{\operatorname{FS}}
\def\FSk{\FS_k}
\def\Hilb{\operatorname{Hilb}}

\def\beq{\begin{equation}}
\def\eeq{\end{equation}}
\def\lb{\label}
\hypersetup{
    bookmarks=true,         
    unicode=false,          
    pdftoolbar=true,       
    pdfmenubar=true,       
    pdffitwindow=false,    
    pdfstartview={FitH},   
    pdftitle={Quantization in geometric pluripotential theory},    
    pdfauthor={T. Darvas, C.H. Lu, Y.A. Rubinstein},     
    colorlinks=true,       
   linkcolor=black,          
    citecolor=black,        
    filecolor=black,      
    urlcolor=black}           

\title{Quantization in geometric pluripotential theory}

\author{Tam\'as Darvas, Chinh H. Lu, Yanir A. Rubinstein}
\date{}
\begin{document}
\maketitle
\begin{abstract}

The space of K\"ahler metrics can, on the one hand, be approximated by subspaces of algebraic metrics, while, on the other hand, can be enlarged to finite-energy spaces arising in pluripotential theory. The latter spaces are realized as metric completions of Finsler structures on the space of K\"ahler metrics. The former spaces are the finite-dimensional spaces of Fubini--Study metrics of  K\"ahler quantization. The goal of this article is to draw a connection between the two.

We show that the Finsler structures on the space of K\"ahler potentials can be quantized. More precisely, given a K\"ahler manifold
 polarized by an ample line bundle we endow the space of Hermitian metrics on powers of that line bundle with Finsler structures and show that the resulting path length metric spaces recover the  corresponding metric completions of the Finsler structures on the space of K\"ahler potentials. This has a number of applications, among them a new approach to the rooftop envelopes and Pythagorean formulas of K\"ahler geometry, a new Lidskii type inequality on the space of K\"ahler metrics, and approximation of finite energy potentials, as well as geodesic segments by the corresponding smooth algebraic objects.
\end{abstract}

\section{Introduction}

Given an ample line bundle
$L$ over a compact \K manifold $(X,\o)$, 
a major theme in K\"ahler geometry, going back to a problem of Yau \cite[p. 139]{Yau2} and work of Tian  \cite{Ti88, Ti90}  thirty years ago, 
has been  the approximation (or ``quantization'')
of the infinite-dimensional space of K\"ahler potentials
$$\mathcal H_\o:= \{u \in C^\infty(X) \,:\,\,  \o_u:=\o + \i\ddbar u >0 \},$$
by the finite-dimensional spaces 
$$\mathcal H_k := \{ \textup{positive Hermitian forms on } H^0(X,L^k)  \},$$
since the $\mathcal H_k$ can be identified as {\it subspaces} of $\mathcal H_\o$
consisting of (algebraic) Fubini--Study metrics. 
Around the same time,  Mabuchi and Semmes introduced an $L^2$ type metric on $\calH_\o$
\cite{Ma87,Semmes88,Se92} and about fifteen years ago it was suggested  by Donaldson that
the geometry of $\calH_\o$ should be approximated by the geometry of $\calH_k$ \cite[p. 483]{Do01}. 

On the other hand, recently $L^p$ type Finsler structures on $\mathcal H_\o$ were introduced, and 
the resulting path length metric structures $(\mathcal H_\o,d_p)$ along with their completions 
$$(\mathcal E^p_\o,d_p)$$ were studied in depth by the first-named author \cite{Da15}. 
Subsequently, the space $(\mathcal E^1_\o, d_1)$
was shown to be intimately related to  existence of special metrics and 
energy properness,
leading to a proof of long-standing conjectures of Tian on the analytic characterization of K\"ahler--Einstein metrics and the strong 
Moser--Trudinger inequality \cite{DR17} (see
Tian--Zhu \cite{TZ99} and Phong et al. \cite{PSSW}
for important earlier progress), paving the way
to a number of related advances on key problems
concerning further relations betweem stability and existence of 
canonical metrics \cite{BBJ15, BDL16,BDL17, DeR17, SD17, CC1,CC2}.
For additional references we refer to the recent surveys \cite{Bo18,Da17,R}.

These two approaches to studying the space $\calHo$---quantization vs. 
completion---have been considered complementary in the literature.
The main goal of this article is to draw a connection between the two
and approximate/quantize the spaces 
$\mathcal E^p_\o$ 
of geometric pluripotential theory
by the finite-dimensional spaces $\mathcal H_k$ of K\"ahler quantization.
As we will see, this leads to new results on both ``sides".  

Furthermore, this is relevant, if not central, to 
the variational program on the Yau--Tian--Donaldson conjecture
that involves showing that geometric data arising in $\mathcal E^p_\o$ 
can be approximated using algebro-geometric data from $\mathcal H_k$, as $k \to \infty$ \cite{Bo18}.

More specifically, in this paper we introduce  $L^p$ Finsler geometries on $\mathcal H_k$ and show that the resulting (complete) path length metric spaces $$(\mathcal H_k,d_{p,k})$$ approximate {\it and recover} $(\mathcal E^p_\o,d_p)$  in the large  $k$ limit.

Since general elements of $(\mathcal E^p_\o,d_p)$ are rather singular (in particular, not even bounded), this significantly extends and sheds light on the broader geometric meaning of 
the work of Berndtsson, Chen-Sun, Phong-Sturm, and Song--Zelditch \cite{Bern09,CS12,PS06,SZ}, also complementing classical results going back to Tian, Bouche, and Kempf \cite{Ti90,Bo90,Ke92} with improvements by Ruan, Catlin, Zelditch, Lu,
and  Ma--Marinescu \cite{Ruan,Ca99,Ze98,Lu00,MM}. All of these works only considered approximation of smooth potentials (i.e., elements of $\calH_\o$), within the context of the $L^2$ geometry.

\paragraph{The $L^p$ Finsler geometry of Hermitian matrices.} In the first part of the article we introduce different $L^p$ Finsler structures on 
$$
\Pn:=\{\h{positive Hermitian $n$-by-$n$ matrices}\}.
$$
For any $h\in\Pn$, the tangent space is
\beq\lb{Hermneq}
T_h\Pn=\Hermn:=\{\h{Hermitian $n$-by-$n$ matrices}\}.
\eeq
There is a classical Riemannian metric on $\Pn$,
$$
\langle \eta,\nu \rangle|_h:=\frac1n \tr \big[h^{-1}\eta h^{-1}\nu \big],\qq \eta,\nu\in T_h\Pn, 
$$
and, by a standard variational argument \cite[p. 195]{Kob}, geodesics
with endpoints $h_0,h_1 \in \mathcal P_n$ are solutions of
\begin{equation}
\label{eq: goed_eq_intr}
\frac{d}{dt} \Big(h_t^{-1} \cdot \dot{h}_t\Big)=0, \ \ \ t \in [0,1],
\end{equation}
thus (op. cit.),
$$
d_{2,\Pn}(h_0,h_1)= \bigg[\frac{1}{n}\sum_{j=1}^n |\lambda_j|^2 \bigg]^{\frac{1}{2}}, 
$$
where 
\begin{equation}
\label{eq: eigenvals_intro}
e^{\lambda_1}, \ldots, e^{\lambda_n}
\end{equation}
 are the eigenvalues of $h_0^{-1}h_1$. 
For any $p\ge1$ we introduce Finsler structures on $\Pn$,
$$
||\nu||_{p,h}:=
\bigg[\frac1n\tr\big(|h^{-1}\nu|^p\big)\bigg]^{\frac{1}{p}},\qq \nu\in T_h\Pn. 
$$
We denote by 
$d_{p,\Pn}$
 the resulting path length metric on $\Pn$.
\begin{theorem}\label{thm: Hn_+_metric_geod-p-case} 
Let $p \geq 1$. 
Solutions of (\ref{eq: goed_eq_intr}) 
are metric geodesics of $(\Pn,d_{p,\Pn})$, thus (recall (\ref{eq: eigenvals_intro}))
\begin{equation}\label{eq: def_L^p_metric_intr}
d_{p,\Pn}(h_0,h_1)= \bigg[\frac{1}{n}\sum_{j=1}^n |\lambda_j|^p \bigg]^{\frac{1}{p}}, \ \  h_0,h_1 \in \Pn,
\end{equation}
and therefore $(\Pn,d_{p,\Pn})$ is a geodesic metric space.
\end{theorem}

Note that the geodesic equation is therefore {\it independent of $p$}.
This result is inspired  by an analogous result
in the infinite-dimensional setting of $\Ho$: the $L^p$ Finsler structures on $\Ho$
have common geodesics \cite[Theorem 1]{Da15}. 
The proof of Theorem \ref{thm: Hn_+_metric_geod-p-case} is also
inspired by that of op. cit., and actually
requires proving a stronger result (Theorem \ref{thm: Hn_+_metric_geod}) showing that the much larger class of Finsler structures defined by weights $\chi$
with $\chi \in \mathcal W^+_p$
have common geodesics. Here, $\mathcal W^+_p$ is a well-studied
class of strongly convex Orlicz weights that have growth 
properties somewhat similar to those of an $L^p$ weight, but are more flexible (see Definition  \ref{WpDef}).
The $L^p$ Finsler norms are not differentiable in general and so our proof
of Theorem \ref{thm: Hn_+_metric_geod-p-case} requires this larger generality.
We mention that shortly before posting this article, the work \cite{BE18} appeared, 
where the authors also considered Finsler metrics on $
\mathcal H_k$, and obtained a general result that is similar to Theorem  \ref{thm: Hn_+_metric_geod-p-case}, motivated by completely different questions in non-Archimedean geometry
\cite[Theorem 3.7]{BE18}.

\paragraph{Quantization of the $L^p$ Mabuchi geometry.} 
We move on to the main topic of the article, the quantization of the complete metric spaces $(\mathcal E^p_\o,d_p)$. Let
$$\textup{d}_k = \dim H^0(X,L^k),$$
and endow the spaces $\mathcal H_k$ with the metric 
\begin{equation}\label{eq: d^k_p_def_intr}
d_{p,k}(\cdot,\cdot) :=\frac{1}{k} d_{p,\P_{\textup{d}_k}} \hspace{-0.07cm}(\cdot,\cdot).
\end{equation}
Here, $\frac{1}{k}$ plays the role of the Planck constant in quantum mechanics.  
Note that this definition is legitimate, as \eqref{eq: def_L^p_metric_intr} implies the metric $d_{p,\P_{\textup{d}_k}}$ is invariant under
unitary change of basis  in $\Bbb C^{\textup{d}_k}$.

As shown in \cite{Da15}, and recalled in  Section 2.2 below, 
the metric completion $\overline{(\mathcal H,d_p)}$ can be identified with $(\mathcal E^p_\o,d_p)$, where $\mathcal E^p_\o$ 
is a subset of $\textup{PSH}(X,\o)$ introduced by Guedj--Zeriahi \cite{GZ07}. The main goal of this work is to recover these infinite-dimensional complete metric structures by taking some large limit of $(\mathcal H_k,d_{p,k})$.
To that end, let us first recall the basic maps between
the finite-dimensional and infinite-dimensional spaces.

The first map is the less obvious one and goes
from $\mathcal E^p_\o$ to $\mathcal H_k$. We fix $h_L$, a hermitian metric on $L$ whose curvature is $\o=\Theta(h_L)=-\i\ddbar \log h_L>0$. We denote by $h_L^k$ the $k$-th tensor product of the metric on $L^k$, the $k$-th tensor product of  $L$. 

Define the Hilbert map
${{\hbox{\rm H}}}_k:\mathcal E^p_\o \to \mathcal H_k$ by
$${{\hbox{\rm H}}}_k(u)(s,s): = \int_X h_L^k(s,s)e^{-ku}\o^n.$$
This is not quite the well-known Hilbert map often denoted by $Hilb_k$  in the literature
\cite[\S2]{Do05}, since we integrate against $\o^n$ instead of $\o_u^n$. 
In fact, for any $u \in \mathcal E^p_\o \setminus \mathcal E^q_\o, \ q >p,$ the integral $\int_X e^{-ku}\o_u^n$ is seen to be infinite for all $k$.
On the other hand, since elements of $\mathcal E^p_\o$ have zero Lelong numbers,
the map ${{\hbox{\rm H}}}_k$ above is well-defined.
 In particular, in the context of the $L^p$ metric completions $\mathcal E^p_\o$, this definition of ${{\hbox{\rm H}}}_k$ is not only the most natural, but also 
the only one that makes sense.  

In the opposite direction, the classical map ${\FSk}: \mathcal H_k \to \mathcal H_\o \subset \mathcal E^p_\o$ sends  an inner product $G$ to the associated Fubini--Study metric restricted to $X$, 
$${\FSk}(G): = \frac{1}{k}\log \sum_{j=1}^{\textup{d}_k} |e_j|^2_{h_L^k},$$
where $\{ e_j\}_{j=1,\ldots,\textup{d}_k}$ is a (any) $G$--orthonormal basis of $H^0(X,L^k).$ 
Equivalently, $\FSk(G)$ can be thought of as a Bergman kernel for which the classical
extremal characterization will prove handy:
\begin{equation}\label{eq: FSk_extremal_def}
{\FSk}(G)(x) = \sup_{s \in H^0(X,L^k), G(s,s)=1} \frac{1}{k} \log |s(x)|^2_{h_L^k}.
\end{equation}

The next theorem summarizes our main results, as we quantize the points, distances and geodesics in $(\mathcal E^p_\o,d_p)$.

\begin{theorem}\label{thm: L^p_Quant_intr} For any $p \geq 1$ the following hold:\\ 
(i)(Quantization of points) For $v \in \mathcal E^p_\o$ we have $$\lim_k d_p({\FSk} \circ {{\hbox{\rm H}}}_k(v),v) =0.$$ 
\noindent (ii)(Quantization of distance) For $v_0,v_1 \in \mathcal E^p_\o$ we have $$\lim_k d_{p,k}({{\hbox{\rm H}}}_k(v_0), {{\hbox{\rm H}}}_k(v_1)) = d_p(v_0,v_1).$$
\noindent (iii)(Quantization of geodesics) Suppose $u_0,u_1 \in \mathcal E^p_\o$ and $[0,1] \ni t \to u_t \in \mathcal E^p_\o$ is the $L^p$-finite-energy geodesic connecting $u_0,u_1$ (Definition \ref{feDef}). Let $[0,1] \ni t \to U^k_t \in \mathcal H_k$ be the $L^p$-Finsler geodesic  joining $U_0^k={{\hbox{\rm H}}}_k(u_0)$ and $U_1^k={{\hbox{\rm H}}}_k(u_1)$, solving (\ref{eq: goed_eq_intr}). Then 
$$\lim_k d_p({\FSk}(U^k_t), u_t) = 0 \ \textup{ for any } \ t \in [0,1].$$
\end{theorem}

Theorem \ref{thm: L^p_Quant_intr} (i) is 
the geometric pluripotential theory analogue of the asymptotic expansion of the smooth Bergman kernel
(i.e., for {\it smooth $v$}) due to the work of Boutet de Monvel--Sjostrand, Catlin, Tian, and Zelditch
 \cite{BS76,Ca99,Ti90, Ze98} (for convergence to equilibrium in case of non-positive metrics  see \cite{Berm09, DMM16}).
 
Similarly, part (ii) for {\it smooth} $v_0,v_1$ and $p=2$ is a result of  Chen-Sun \cite[Theorem 1.1]{CS12}, using a slightly different ${{\hbox{\rm H}}}_k$ map  
(with an alternative proof due to Berndtsson \cite[Theorem 1.1]{Bern09} {using the language of spectral measures}). 

Finally, part (iii) for {\it smooth}  $u_0,u_1$ and $p=2$ is a theorem of Berndtsson \cite[Theorem 1.2]{Bern09}, {extending previous work of Phong-Sturm \cite[Theorem 1]{PS06}}. For smooth potentials $u_0,u_1$, Berndtsson proves actually $C^0$-convergence of $FS_k(U^k_t)$, which implies $d_2$-convergence by \cite[Theorem 3]{Da15}. Similarly, in case of toric manifolds, Song-Zelditch prove $C^2$-convergence of $FS_k(U^k_t)$.
We emphasize though that the $d_p$-convergence in our result is optimal since a typical element of $\mathcal E^p_\o$ is unbounded. 

Compared to the above mentioned works, in the absence of smoothness, the well known asymptotic expansion of the smooth Bergman kernel will have only very limited use, and instead we will have to rely almost exclusively on pluripotential theoretic and complex-algebraic tools. In addition to techniques in finite-energy pluripotential theory, our two cornerstones are the Ohsawa--Takegoshi extension theorem \cite{OT87} and the quantized maximum principle of Berndtsson \cite{Bern09}. However to use these latter theorems, one needs to work with strongly positive currents. In general finite-energy currents do not satisfy this  positivity property, and we need to develop a suitable approximation technique using strongly positive currents, this being one of the novelties of this work. 

Finally, we note that Theorem \ref{thm: L^p_Quant_intr} collectively answers questions raised by Guedj in the case $p=2$ \cite[p. 2]{Gu14}. Also, since K\"ahler metrics with Poincar\'e type singularities are contained in  $\mathcal E^p_\o$ for any $p \geq 1$, one would think that our global approximation results are possibly connected with results of local nature by Auvray--Marinescu--Ma \cite{AMM16} on Poincar\'e type metrics.

\paragraph{Quantization of rooftop envelopes and the Pythagorean formula.} We end the article by quantizing the rooftop envelopes and Pythagorean formulas of $\mathcal E^p_\o$, both essential ingredients in the metric geometry of $(\mathcal E^p_\o,d_p)$.

Given $u,v \in \mathcal  E^p_\o$, define
$$P(u,v):=\sup \{w \in \textup{PSH}(X,\theta) \textup{ \  s.t. \ } w \leq u,v \}.$$
Then  $P(u,v) \in \mathcal E^p_\o$ \cite[Theorem 3]{Da17}.
This rooftop type envelope was introduced by Ross--Witt Nystrom \cite{RWN14}, and its regularity was studied by two of us \cite{DR16}. The indispensable role of the operator $(u,v) \to P(u,v)$ within $L^p$ Mabuchi geometry was pointed out in \cite{Da15,Da17}, summarized by the {\it Pythagorean formula}:
\begin{equation}\label{eq: Pythagorean_intr}
d_p(u,v)^p = d_p(u,P(u,v))^p + d_p(v,P(u,v))^p. 
\end{equation} 
In the historically important $p=2$ case, this formula simply says that  $u,v$ and $P(u,v)$ form a right triangle with hypotenuse $uv$, motivating the origin of the name. This formula played a pivotal role in proving that $(\mathcal E^p_\o,d_p)$ is complete, by showing that an arbitrary Cauchy sequence is equivalent to a Cauchy sequence whose potentials are additionally monotone increasing. In addition to this, one can give an explicit formula for $d_1(u,v)$ in terms of $P(u,v)$ and the 
Monge--Amp\`ere energy, leading to equivalence of $d_1$-properness and $J$-properness, which paved the way to a proof of long-standing conjectures of Tian
\cite{DR17} and subsequently 
a number of related advances linking existence of canonical metrics to energy properness and stability \cite{BBJ15, BDL16, CC1, CC2}. 

Very recently, rooftop envelopes and Pythagorean formulas have been considered in the context of non-Archimedean K\"ahler geometry as well \cite{BJ18}, and it would be interesting to explore the connection between our Theorem 1.3 below and the results of \cite{BJ18}. 

Despite the numerous applications, until now the origin of the Pythagorean formula
remained mysterious. We now show that in fact this equation is the quantized version of an elementary metric identity for Hermitian matrices. 

Given $h_0,h_1 \in \mathcal H_k$ it is possible to find a $h_0$-orthonormal basis with respect to which
$$h_0 = \textup{diag}(1,\ldots,1) \ \textup{ and } \ h_1=\textup{diag}(e^{\lambda_1}, \ldots,e^{\lambda_{\textup{d}_k}})$$
We introduce the \emph{quantum rooftop envelope} as follows:
$${{P}}_k(h_0,h_1):= \textup{diag}(\max(1,e^{\lambda_1}),\ldots,\max(1,e^{\lambda_{\textup{d}_k}}))=\textup{diag}(e^{\lambda_1^+},\ldots,e^{\lambda_{\textup{d}_k}^+}).$$
We see that ${{P}}_k(h_0,h_1)$ is well-defined and invariant under change of $h_0$-orthonormal bases. Moreover, somewhat imprecisely one may think of ${{P}}_k(h_0,h_1)$ as the smallest Hermitian form that is bigger than both $h_0$ and $h_1$. Comparing with \eqref{eq: def_L^p_metric_intr}, it is elementary to verify that the quantum Pythagorean formula holds:
\begin{equation}\label{eq: Pyth_formula_quant_intr}
d_{p,k}(h_0,h_1)^p = d_{p,k}(h_0,{{P}}_k(h_0,h_1)))^p + d_{p,k}({{P}}_k(h_0,h_1),h_1)^p.
\end{equation}
In our last main result we simultaneously quantize the rooftop envelope $P(u,v)$ together with the Pythagorean formula:
\begin{theorem}\label{thm: Quant_Pyt_rooftop_intr} Let $u_0,u_1 \in \mathcal E^p_\o$. Then the following hold:\\
\noindent (i) $\lim_k d_p(P(u_0,u_1), {\FSk}({{P}}_k({{\hbox{\rm H}}}_k(u_0), {{\hbox{\rm H}}}_k(u_1)))) =0.$ \\
\noindent (ii) $\lim_k d_{p,k}\big({{\hbox{\rm H}}}_k(u_0), {{P}}_k({{\hbox{\rm H}}}_k(u_0),{{\hbox{\rm H}}}_k(u_0)) \big)=d_p(u_0,P(u_0,u_1)).$
\end{theorem}

This result is already new in the particular case of smooth potentials. Moreover, 
Theorem \ref{thm: Quant_Pyt_rooftop_intr} (ii) together with Theorem \ref{thm: L^p_Quant_intr} (ii) implies that after composing with ${{\hbox{\rm H}}}_k$, all three  expressions in the quantum Pythagorean formula \eqref{eq: Pyth_formula_quant_intr} converge to the corresponding terms of \eqref{eq: Pythagorean_intr}.

An important step in arguing this last theorem consists of quantizing (via Theorem \ref{thm: L^p_Quant_intr}) an inequality of Lidskii for Hermitian matrices,  giving a new result about the metric geometry of K\"ahler potentials (Theorem \ref{thm: mon_rev_triang_ineq}):  if $u,v,w \in \mathcal E^p_\o$ satisfy $u \geq v \geq w$,  then
$$d_p(v,w)^p \leq d_p(u,w)^p - d_p(u,v)^p.$$

To end this introduction, we remark that for concreteness, 
we only focus on the $L^p$ Finsler structures on both $\mathcal H_k$ and $\mathcal H_\o$ (with the exception of Section 3, where smooth Orlicz--Finsler metrics need to be considered for sake of approximation). However, all our results extend to the Orlicz--Finsler setting as well, as considered in \cite{Da15}. We leave it to the interested reader to adapt our arguments to that setting.

\paragraph{Acknowledgments.} Work supported by BSF grant 2016173 and NSF grants DMS-1515703, 1610202. We thank D. Coman, L. Lempert and G. Marinescu for their remarks on a preliminary version of this paper.

\section{Background and preliminary results}

To fix terminology, for the duration of the paper $(L,X)$ is a line bundle with Hermitian metric $h_L$, whose total curvature is equal to  $\o$, the background K\"ahler metric on $X$. With regards to the total volume of our class, we introduce the following recurring quantity:
$$V := \int_X \o^n.$$

\subsection{The finite-energy spaces $\mathcal E^p_\o$}

In this short subsection we recall the basics of finite energy pluripotential theory, as introduced by Guedj-Zeriahi \cite{GZ07}. For a detailed account on these matters we refer to the recent textbook \cite{GZ17}.

By $\textup{PSH}(X,\omega)$ we denote the space of $\omega$-plurisubharmonic ($\omega$-psh) functions. Extending a notion of Bedford-Taylor, Guedj-Zeriahi introduced the non-pluripolar Monge-Amp\`ere mass of a potential $u \in \textup{PSH}(X,\omega)$ as the following limit \cite{GZ07}:
$$\omega_u^n := \lim_{k \to \infty}\mathbbm{1}_{\{u > -k\}} (\omega + \i\ddbar \max(u,-k))^n.$$
For such measures one has a bound $\int_X \o_u^n \leq \int_X \o^n=V$, and $\mathcal E_\o$ is the set of potentials with full mass:
$$\mathcal E_\o = \Big\{ u \in \textup{PSH}(X,\omega) \ \textup{s.t.} \ \int_X \o_u^n = \int_X \o^n=V\Big\}.$$
Furthermore, potentials $u \in \mathcal E_\o$ that satisfy $L^p$ type integrability condition $p \geq 1$ are members of the \emph{finite-energy spaces}:
$$\mathcal E^p_\o = \Big\{ u \in \mathcal E \ \textup{s.t.} \  \int_X |u|^p \o_u^n < +\infty\Big\}.$$
The \emph{fundamental inequality} of $\mathcal E^p_\o$ is as follows:
\begin{lemma}\textup{\cite[Lemma 2.3]{GZ07}} \label{lem: fund_ineq} Suppose $u,v \in \mathcal E^p_\o$ with $u \leq v \leq 0$. Then there exists $C(p) >0$ such that:
$$\int_X |v|^p \omega_v^n \leq C \int_X |u|^p \omega_u^n.$$
\end{lemma}

\begin{lemma}\textup{\cite[Corollary 2.7]{GZ07}} \label{lem: uniform_est} Suppose that $u_k \in \mathcal E^p_\o$ and $u \in \textup{PSH}(X,\omega)$ such that $u_k \searrow u$. Then if $\int_X |u_k|^p\o_{u_k}^n \leq C$ for all $k$, then $u \in \mathcal E^p_\o$ and $\int_X |u|^p\o_{u}^n \leq C$. 
\end{lemma}

The following result is known, but we include a proof.

\begin{lemma}\label{lem: E^1_stability} Suppose $\eta$ is a K\"ahler form. If $u \in \mathcal E^p_\o$ and $u \in \textup{PSH}(X,\eta)$,  then $u \in \mathcal E^p_\eta$.
\end{lemma}

\begin{proof}We can assume that $u \leq 0$ and that $\eta \leq k \o$ for some $k \in \Bbb N^*$.

From the Lemma \ref{lem: fund_ineq} it follows that $\int_X |u_l|^p \o_{u_l}^n \leq C$, where $u_l := \max(u,-l), \ l \geq 0$. Using Lemma \ref{lem: fund_ineq} again, we conclude that $u_l/k \in \mathcal E^p_\o$, more precisely, we have that $\int_X |u_l|^p (k\o + i\ddbar u_l)^n \leq C(k)$. This ultimately implies that 
$$\int_X |u_l|^p \eta_{u_l}^n \leq \int_X |u_l|^p (k\o + i\ddbar u_l)^n <C(k).$$
Using Lemma \ref{lem: uniform_est} we conclude that $u \in \mathcal E^p_\eta$.
\end{proof}

Lastly, we mention a well known convergence result from \cite{GZ07}, which is a particular case of \cite[Proposition 2.11]{Da18}:
\begin{lemma}\label{lem: limit_energy_L^p} Suppose that $u_k,u \in \mathcal E_\o^p$ and $u_k$ decreases/increases a.e. to $u$. Then $$\lim_k \int_X |u_k|^p \o_{u_k}^n = \int_X |u|^p \o_{u}^n.$$
\end{lemma}

\subsection{The $L^p$ Finsler geometry on the space of K\"ahler potentials}

Here we recall some of the main points on the $L^p$ Finsler geometry of the space of K\"ahler potentials. For a detailed exposition, we refer to \cite[Chapter 3]{Da18}, as well as the original articles \cite{Da17,Da15}.

As follows from the definition, the space of K\"ahler potentials $\mathcal H_\o$ is a convex open subset of $C^\infty(X)$, hence one can think of it as a trivial Fr\'echet manifold. As such, one can introduce on $\mathcal H_\o$ a collection of $L^p$ type Finsler metrics. If $u \in \mathcal H_\o$ and $\xi \in T_u \mathcal H_\o \simeq C^\infty(X)$, then the $L^p$-length of $\xi$ is given by the following expression:
$$
\| \xi\|_{p,u} = \bigg(\frac{1}{V}\int_X |\xi|^p \o_u^n\bigg)^{\frac{1}{p}}.
$$
In case $p=2$, this is the Riemannian geometry of Mabuchi \cite{Ma87}
(cf. Semmes \cite{Se92} and Donaldson \cite{Do99}).

Using these Finsler metrics, one can introduce path length metric structures $(\mathcal H_\o,d_p)$. In \cite{Da15}, the completion of these spaces was  identified with $\mathcal E^p_\o \subset \textup{PSH}(X,\o)$, and it turns out that $(\mathcal E^p_\o,d_p)$ is a complete geodesic metric space.

The geodesic segments of the completion $(\mathcal E^p_\o,d_p)$ are constructed as certain upper envelopes of quasi-psh functions as we now detail. Let $S = \{0 < \textup{Re }s < 1 \} \subset \Bbb C$ be the unit strip, and $\pi_{S \times X}: S \times X \to X$ denotes projection to the second component.

We consider $u_0,u_1 \in \mathcal E^p_\o$. We say that the curve $[0,1] \ni t \to v_t \in \mathcal E^p_\o$ is a \emph{weak subgeodesic} connecting $u_0,u_1$ if $d_p(v_t,u_{0,1}) \to 0$ as $t \to 0,1$, and the extension $v(s,x) = v_{\textup{Re }s}(x)$ is $\pi^* \o$-psh on $S \times X$, i.e.,
$$\pi^* \o + \i\partial_{S \times X} \bar \partial_{S \times X} v \geq 0, \ \textup{ as currents on } \ S \times X.$$
As shown in \cite{Da17,Da15}, a distinguished $d_p$-geodesic $[0,1] \ni t \to u_t 
\in \mathcal 
E^p_\o$ connecting $u_0,u_1$ can be obtained as the supremum of all weak subgeodesics:
\beq
\lb{fegeod}
u_t := \sup \{v_t \ | \ t \to v_t  \textup{ is a subgeodesic connecting } u_0,u_1\}, \ t \in [0,1].
\eeq

\begin{definition}
\label{feDef}
Given $u_0,u_1\in \mathcal E^p_\o$, we call \eqref{fegeod}
the $L^p$-finite-energy geodesic connecting $u_0,u_1$.
\end{definition}

In particular, the curve $t \to u_t$ naturally satisfies a maximum/comparison principle. In case the endpoints $u_0,u_1$ are smooth strictly $\omega$-psh, then the weak geodesic connecting them is actually $C^{1\bar 1}$ on $\overline{S} \times X$, as shown by 
Chen \cite{Ch00} (with complements by B\l ocki \cite{Bl12}). 

 With regards to the metric $d_p$ we have the following precise double estimate for some dimensional constant $C>1$ \cite[Theorem 3]{Da15}:
$$\frac{1}{C}d_p(u_0,u_1)^p \leq  \frac{1}{V}\int_X |u_0 - u_1|^p \o_{u_0}^n + \frac{1}{V}\int_X |u_0 - u_1|^p \o_{u_1}^n \leq C d_p(u_0,u_1)^p, \ \ \ u_0,u_1 \in \mathcal E^p_\o.$$

\subsection{The calculus of diagonalizable matrices}

In this short subsection we collect some basic facts concerning the calculus of diagonalizable matrices. Our treatment will be minimalistic and for a more thorough study we refer to Bhatia \cite{Bh97}. Denote by 
$$\Bbb D^n \subset \Bbb C^{n \times n}$$ 
the set of complex valued $n \times n$ \emph{diagonalizable matrices with real eigenvalues}. We remark that the set of Hermitian matrices is contained in $\Bbb D^n$, however we need to work with a bigger class of matrices in the sequel. 

Given a function $f: \Bbb R \to \Bbb R$  one can define an operator $f:\Bbb D^n \to \Bbb R$ in the following manner:
\begin{equation} \label{f_matrix_def}
f(A) = U \cdot \textup{diag}( f(\lambda_1), \ldots, f(\lambda_n)) \cdot U^{-1},
\end{equation}
where $\{\lambda_j\}_{j=1,...,n}$ are the eigenvalues of $A$, and the columns of the $n \times n$ matrix $U$ consist of an eigenbase of $A$.
 
By elementary means one can show that this definition does not depend on the choice of $U$. In particular, in case $f(t)$ is the polynomial $\sum_{j=0}^k a_k t^k$ we get that $f(A)= \sum_{j=0}^k a_k A^k$, as expected. This simple fact will be used multiple times below.

\paragraph{Variation of the trace under matrix functions.} Given that eigenvectors tend to misbehave under small perturbation, it is not immediately clear from the above definition how the differentiability of $f$ is reflected in the Fr\'echet differentiability of $A \to f(A)$. 
However after taking the trace of $A \to f(A)$ one ends up with a familiar looking identity, that we will need later:

\begin{proposition}\label{prop: Tr_f_deriv} Suppose that  $f \in C^1(\Bbb R)$ and $I \subset \Bbb R$ is an open interval . We introduce  $I \ni t \to A_t \in \Bbb D^n$  and Hermitian matrics $I \ni t \to h_t \in \Pn$, both smooth curves, such that each $A_t$ is $h_t$-self-adjoint.  Then $t \to {{\tr}} \big[f(A_t) \big]$ is continuously differentiable and 
\begin{equation}\label{eq: f_deriv_formula}
\frac{d}{dt}  {{\tr}} \big[f(A_t) \big]= {{\tr}} \big[f'(A_t) \cdot \dot A_t\big], \ t \in [0,1].
\end{equation}
\end{proposition}

Smoothness of $t \to A_t,h_t$ simply means that the coefficients of the matrices are smooth functions of $t$. Here we say that $A\in \mathbb{D}^n$ is $h$-self-adjoint (for $h\in \Pn$) if $h\dot A$ is a Hermitian matrix. 

\begin{proof} In case $f$ is a polynomial, an elementary calculation yields \eqref{eq: f_deriv_formula}, using the identity ${{\tr}}\big[ V\cdot W \big]={{\tr}}\big[ W\cdot V \big], \ V,W \in \Bbb C^{n \times n}.$

We will now argue \eqref{eq: f_deriv_formula} for general $f \in C^1(\Bbb R)$. By the Stone-Weierstrass theorem we can find polynomials $f_k$ that approximate $f$ and $f'$ uniformly on compact subsets of $\Bbb R$. Fixing $t \in I$ and $h \in \Bbb R$ small, we have the following identity:
\begin{equation}\label{eq: f_k_id}
{{\tr}} \big[f_k(A_{t+h}) \big]-{{\tr}} \big[f_k(A_{t}) \big]= \int_t^{t+h}{{\tr}} \big[f'_k(A_l) \cdot \dot A_l\big]dl.
\end{equation}
By definition, we get that ${{\tr}} \big[f_k(A_{t+h}) \big] \to {{\tr}} \big[f(A_{t+h}) \big]$ and ${{\tr}} \big[f_k(A_{t}) \big] \to {{\tr}} \big[f(A_{t}) \big]$. 

To prove that the integral on the right hand side also converges, we will use the dominated convergence theorem, and the fact that each $A_t$ is $h_t$-self-adjoint. Indeed, for each $l \in [t,t+h]$ it is possible to find an $h_l$-unitary matrix $U_l$ such that
\begin{equation}\label{eq: diag_U_unitary}
A_l = U_l \cdot \textup{diag}(\lambda_1^l,\ldots,\lambda_n^l) \cdot {U_l}^{-1}, \ \ \ \ f'_k(A_l) = U_l\cdot \textup{diag}(f'_k(\lambda_1^l),\ldots,f_k'(\lambda_n^l)) \cdot {U_l}^{-1}.
\end{equation} 
If we write $A_t=h_l^{-1}B_l=h_l^{-1/2}(h_l^{-1/2}B_lh_{l}^{-1/2})h_l^{1/2}$, where $B_l=h_lA_l$ is a Hermitian matrix, then it is clear that the coefficients of $U_l, {U_l}^{-1}, l \in [t,t+h]$ need to be uniformly bounded. As $l \to A_l$ is continuous, we obtain that the eigenvalues of $A_l, \ l \in [t,t+h]$ are bounded too, hence so are the eigenvalues of $f'_k(A_l)$ (since $f'_k$ converges to $f'$ uniformly on compact sets of $\mathbb{R}$). Putting all these together, we obtain that ${{\tr}} \big[f'_k(A_{l}) \cdot \dot A_l \big], l \in [t,t+h]$ is uniformly bounded, hence we can use the dominated convergence theorem in \eqref{eq: f_k_id} to conclude that
$${{\tr}} \big[f(A_{t+h}) \big]-{{\tr}} \big[f(A_{t}) \big]= \int_t^{t+h}{{\tr}} \big[f'(A_l) \cdot \dot A_l\big]dl.$$
To finish the proof of \eqref{eq: f_deriv_formula}, we only have to argue that $l \to f'(A_l) \cdot \dot A_l$ is continuous on $[t,t+h]$. This is a consequence of the fact that each  
$l \to f'_k(A_l) \cdot \dot A_l$ is continuous, and $f'_k(A_l) \cdot \dot A_l$ converges uniformly to $f'(A_l) \cdot \dot A_l$ on $[t,t+h]$, as follows from $\eqref{eq: diag_U_unitary}$.
\end{proof}

\paragraph{An inequality of Lidskii.} Here we present a corollary to an inequality of Lidskii that is likely known, but we could not find a reference to it in the literature. This estimate will be used in the proof of Theorem \ref{thm: Quant_Pyt_rooftop_intr}:

\begin{theorem} \label{thm: Lidski_cor} Suppose that $A,B \in \Pn$ are such that $I_n \leq A \leq B$ and let $p \geq 1$. Then
\begin{equation}\label{eq: Lidskii_cor_est}
{{\tr}} \big[(\log A^{-1}B)^p\big]  + {{\tr}}\big[ (\log A\big)^p \big]\leq {{\tr}} \big[(\log B\big)^p\big].
\end{equation}
\end{theorem} 

This result can be viewed as a nonlinear generalization of the following elementary fact: given $a,b \in \Bbb R$ such that $0 \leq a \leq b$ the estimate $(b-a)^p +a^p \leq b^p$ holds. Before we can give the proof of this result, we need to recall basic facts about majorization of vectors and doubly stochastic matrices. In our brief presentation we will closely follow \cite[Chapter II--III]{Bh97}, and we refer the interested reader to this work for more details. 

Given $x= (x_1,\ldots,x_n) \in \Bbb R^n$ we will denote by $x^{\uparrow},x^{\downarrow} \in \Bbb R^n$ the vectors whose coordinates are obtained after rearranging the coordinates of $x$ in increasing/decreasing order respectively.

We say that $x \prec y$ if 
$$\sum_{j=1}^k x^{\downarrow}_j \leq \sum_{j=1}^k y^{\downarrow}_j, \ \ 1 \leq k \leq n-1 \ \textup{ and } \ \sum_{j=1}^n x^{\downarrow}_j = \sum_{j=1}^n y^{\downarrow}_j.$$
It is easy to see that $\prec$ is an ordering relation and we say that $x$ is \emph{majorized} by $y$ in case $x \prec y$. As a typical example, keep in mind that if $x_j \geq 0$ and $\sum_{j=1}^n x_j=1$ then 
$$\Big(\frac{1}{n}, \ldots,  \frac{1}{n}\Big) \prec (x_1,\ldots,x_n) \prec (1,0,\ldots,0).$$  
Majorization is intimately related to doubly stochastic transformations as we now point out. We say that $A \in \Bbb  R^{n \times n}$ is doubly stochastic if $a_{ij} \geq 0$ and
$$\sum_{k=1}^n a_{jk} =1, \ \ \ \sum_{k=1}^n a_{kj} =1, \ \ 1 \leq j \leq n.$$
We recall the following fundamental fact:
\begin{proposition}\textup{\cite[Theorem II.1.10]{Bh97}} For $x,y \in \Bbb R$ we have that $x \prec y$ if and only if $x = Ay$ for some doubly stochastic matrix $A \in \Bbb  R^{n \times n}$. 
\end{proposition}

Using this we argue the next elementary lemma:
\begin{lemma}\label{lem: convex_est_Lidskii} \textup{\cite[Theorem II.3.1]{Bh97}} Let $x, y \in \Bbb R^n$ and $\varphi: \Bbb R \to \Bbb R$ convex. Then $x \prec y$ implies that 
$$\sum_{j=1}^n \varphi(x_j) \leq \sum_{j=1}^n \varphi(y_j).$$
\end{lemma}
\begin{proof} The relation $x \prec y$ implies that $x = Ay$ for a doubly stochastic matrix $A$. In particular, due to the convexity of $\varphi$, the following estimate holds:
\begin{flalign*}
\sum_{j=1}^n \varphi(x_j) & = \sum_{j=1}^n \varphi \left (\sum_{k=1}^n a_{jk} y_k \right)  \leq \sum_{j,k} a_{jk} \varphi( y_k) = \sum_{k=1}^n \varphi( y_k).  
\end{flalign*}
\end{proof}

Lastly, before proving Theorem \ref{thm: Lidski_cor}, we recall the following inequality of Lidskii:

\begin{theorem}\textup{\cite[Corollary III.4.6]{Bh97}}\label{thm: Lidskii} Given $A,B \in \Pn$, all the eigenvalues of $AB \in \Bbb C^{n \times n}$ are positive and 
$$\log \big(\lambda ^{\downarrow} A\big) + \log \big(\lambda ^{\uparrow} B\big) \prec \log \lambda (AB).$$
\end{theorem}
In the above statetement $\lambda (C)$ is simply the vector containing the eigenvectors of a matrix $C$. Also, for any function $f: \Bbb R \to \Bbb R$ and vector $v=(v_1,\ldots,v_n) \in \Bbb R^n$, by $f(v) \in \Bbb R^n$ we understand the vector $(f(v_1),\ldots,f(v_n))$.  
\begin{proof}[Proof of Theorem \ref{thm: Lidski_cor}] Notice that \eqref{eq: Lidskii_cor_est} is equivalent to
\begin{equation}\label{eq: Lidskii_cor_est_modified}
{{\tr}} \big[\big(\log A^{-\frac{1}{2}}B A^{-\frac{1}{2}}\big)^p\big]  + {{\tr}} \big[\big(\log A\big)^p\big]\leq {{\tr}}\big[ \big( \log B\big)^p\big].
\end{equation}
Indeed, since $A^{-1}B = A^{-\frac{1}{2}} \big(A^{-\frac{1}{2}}B A^{-\frac{1}{2}}\big)A^{\frac{1}{2}}$, it follows that $A^{-1}B$ and $A^{-\frac{1}{2}}B A^{-\frac{1}{2}}$ have the same real postive eigenvalues. 
To argue this estimate we turn to Lidskii's inequality. Indeed, Theorem \ref{thm: Lidskii} implies that 
$$\log \lambda ^{\downarrow} \big(A^{-\frac{1}{2}}B A^{-\frac{1}{2}}\big) + \log \lambda ^{\uparrow} \big(A\big) \prec \log \lambda (A^{-\frac{1}{2}}B A^{\frac{1}{2}})=\log \lambda (B).$$
Applying Lemma \ref{lem: convex_est_Lidskii} with $\varphi(t) = |t|^p$ to this estimate we obtain that
$$\sum_{j=1}^n\Big|\log \lambda ^{\downarrow}_j  \big( A^{-\frac{1}{2}}B A^{-\frac{1}{2}}\big) + \log \lambda^{\uparrow}_j \big(A\big)\Big|^p \leq  \sum_{j=1}^n \big|\log \lambda_j (B )\big|^p.$$
Since $I_n \leq A^{-\frac{1}{2}}B A^{-\frac{1}{2}}$, and also $I_n \leq A\leq B$, all the eigenvalues involved in the above estimate are greather than $1$, allowing us to get rid of the absolute values and to finish the proof \eqref{eq: Lidskii_cor_est_modified} in the following manner:
\begin{flalign*}
{{\tr}} \big[(\log( A^{-\frac{1}{2}}B A^{-\frac{1}{2}})\big)^p \big]  + {{\tr}} \big[(\log A\big)^p \big]& = \sum_{j=1}^n \Big[\left (\log \lambda ^{\downarrow}_j  \big( A^{-\frac{1}{2}}B A^{-\frac{1}{2}}\big) \right)^p + \left (\log \lambda^{\uparrow}_j (A)\right)^p  \Big] \\ 
&\leq \sum_{j=1}^n \Big[\log \lambda ^{\downarrow}_j  \big( A^{-\frac{1}{2}}B A^{-\frac{1}{2}}\big) + \log \lambda^{\uparrow}_j (A) \Big]^p \\
& \leq   \sum_{j=1}^n \left (\log \lambda_j (B)\right)^p = {{\tr}}\big[(\log(B))^p \big],
\end{flalign*}
where in the first inequality we used that $a^p + b^p \leq (a+b)^p$, for any $a,b\geq 0$.
\end{proof}

\subsection{The Ohsawa--Takegoshi extension theorem and the quantized maximum principle}

We recall a particular case of the Ohsawa--Takegoshi theorem \cite{OT87} that we will use in this article. This result has been extended in many different directions (see \cite{Dem12,BL14} and references therein), yet we could not find an exact reference to the version below, that is well known to experts. As a courtesy to the reader we present a simple proof  using only the classical the Ohsawa--Takegoshi theorem in $\Bbb C^n$ and the classical H\"ormander estimates. For a statement involving smooth metrics we refer to \cite[Proposition 8.8]{BeKe12}, and for more recent results that can be obtained using techniques of similar nature we refer to \cite[Section 5]{CM15} and \cite[Section 3]{CMM17}.

\begin{theorem}\label{thm: OhsTak} Suppose $u \in \textup{PSH}(X,\o)$ such that $\o_u \geq \varepsilon \o$ for some $\varepsilon >0$.  There exists $k_0:=k_0(X,\omega,\varepsilon)$ and $C:=C(X,n,\o)>0$ such that for any  $x \in X$ and $k \geq k_0$ there exists $s \in H^0(X,L^k)$ satisfying $s(x) \neq 0$ and
\begin{equation}\label{eq: OT_ineq}
\int_X h_L^k(s,s)e^{-ku}\o^n \leq C h_L^k(s(x),s(x))e^{-ku(x)}.
\end{equation}
\end{theorem}

\begin{proof}
We fix $x \in X$, and we choose a coordinate ball $B(0,1)$ centered at $x$ that trivializes $L$. By the classical Ohsawa--Takegoshi theorem (see
 \cite[Theorem 3.1]{BL14} for a sharp version), there exists $C':=C'(n)>0$  such that for all for $k \in \Bbb N$ there exists a section $f \in H^0(B(0,1),L^k)$ satisfying $f(0) \neq 0$ and 
\begin{equation}\label{eq: OT_ineq1}
\int_{B(0,1/2)} h_L^k(f,f)e^{-ku} dV \leq C' h_L^k(f(x),f(x))e^{-ku(x)},
\end{equation}
where $dV$ is the Euclidean volume. Next we choose a cutoff function $\eta$, compactly supported inside $B(0,1/2)$, that is identically equal to $1$ on $B(0,1/4)$. The function $z \mapsto \psi(z):=2n \eta(z) \log \|z\|$ is  quasi-plurisubharmonic on $X$ and for $A$ large enough one has $A\omega +i\ddbar \psi  \geq 0$. Since $X$ is compact one can choose a uniform constant $A$ independent of $x$.  Using the classical H\"ormander estimates \cite[Theorem VIII.4.5]{Dem12} 
for $k \geq k_0(\varepsilon,\omega,X)$ we can find $\chi$, a smooth section of $L^k$, such that $\bar \partial \chi = \bar \partial \eta \wedge f$ on $X$ and 
\begin{flalign}\label{eq: OT_ineq2}
\int_X h_L^k(\chi,\chi) & e^{-ku} \o^n   \leq  \int_X h_L^k(\chi,\chi) e^{-ku - 2n\eta \log\|z\|} \o^n \nonumber \\
& \leq \frac{1}{2(\varepsilon k-A)} \int_X \o \otimes h^k  (\bar \partial \eta \wedge f,\bar \partial \eta \wedge f) e^{-ku-2n\eta \log\|z\|} \o^n \nonumber \\
& \leq \frac{C''(\omega,X)}{2(\varepsilon k-A)} \int_{B(0,\frac{1}{2})} h^k (f,f) e^{-ku} dV\nonumber \\ 
&\leq  \int_{B(0,\frac{1}{2})} h^k (f,f) e^{-ku} dV.
\end{flalign}
This integrability condition guarantees that $\chi(x)=0$, and $s:=\eta f - \chi \in H^0(X,L^k)$ by construction. Estimate \eqref{eq: OT_ineq} follows for $s$, after putting together \eqref{eq: OT_ineq1} and \eqref{eq: OT_ineq2}. 
\end{proof}

Next we recall the ``quantized" maximum principle of Berndtsson \cite[Proposition 3.1]{Bern09}. To state this result, 
for an ample line bundle $(E,g) \to X$ we consider  $\mathcal H_{E \otimes K_X}$, the space of Hermitian forms on the vector space of sections $H^0(X,E \otimes K_X)$. 
For $\eta := \Theta(g)=-i\ddbar \log g >0$, we will also adjust the Hilbert map $\mathcal E^p_\eta \to \mathcal H_{E \times K_X}$ in the following manner:
$$
\Hilb_{E \otimes K_X}(v)(s,s) := \int_X g(s,s) e^{-v}, \ \ s \in H^0(X,E \otimes K_X).
$$
As a distinguishing feature of this setup, we note that no choice of volume form is needed to define the right hand side above.

\begin{proposition}\label{prop: quant_max_princ} Let $(E,g) \to X$ be an ample line bundle over $X$, and $\eta = \Theta(g)>0$. Let $[0,1] \in t \to v_t \in \mathcal E_\eta^p$ be a weak subgeodesic connecting $v_0,v_1 \in \mathcal E^p_\eta$. We introduce $V_{v_t} := \Hilb_{E \otimes K_X}(v_t)$, and by 
$$[0,1] \ni t \to V_t \in \mathcal H_{E \otimes K_X}$$ we denote the geodesic connecting $V_0,V_1 \in \mathcal H_{E \otimes K_X}$, solving \eqref{eq: goed_eq_intr}.  If $V_0 \leq V_{v_0}$ and $V_1 \leq V_{v_1}$ then $V_t \leq V_{v_t}, \ t \in [0,1].$
\end{proposition}

Since the original result is only stated for equal boundary values and for subgeodesics $t \to v_t$  that are $C^1$, we give a brief argument showing how the result generalizes to finite-energy subgeodesics and only comparable boundary data.
\begin{proof} We consider $v^l_0,v^l_1 \in \mathcal H_\eta$ decreasing to $v_0,v_1$ respectively. If $t \to v^l_t$ are the $C^{1\bar 1}$-geodesics joining $v^l_0,v^l_1$ then by the comparison principle $v_t\leq v^l_t$. Since $\Hilb_{E \otimes K_X}$ is monotone decreasing it follows that 
$$V_{v^l_t} := \Hilb_{E \otimes K_X}(v_t^l) \leq V_{v_t}.$$  
We apply Berndtsson's maximum principle \cite[Proposition 3.1]{Bern09} to $t \to v_t^l$ and conclude that 
$$V_t^l \leq V_{v^l_t} \leq V_{v_t}, \ t \in [0,1],$$
where $[0,1] \ni t \to V_t^l \in \mathcal H_{E \otimes K_X}$ is the finite-dimensional geodesic joining $V_{v_0^l}$ and $V_{v_1^l}$.

Since solutions to the geodesic equation \eqref{eq: goed_eq_intr} are endpoint stable, we get that $V_t^l \to W_t$ as $l \to \infty$, where $t \to W_t$ is the geodesic joining $V_{v_0}$ and $V_{v_1}$. This gives $W_t \leq V_{v_t}$.

Since $V_0 \leq V_{v_0}=W_0$ and $V_1 \leq V_{v_1}=W_1$, by the well known maximum principle of the finite-dimensional geodesics (see \cite[Chapter 8, Lemma 8.11]{BeKe12}) we conclude that $V_t \leq W_t$, finishing the argument.
\end{proof}

In this work we will apply the above proposition for the line bundle $E_k := L^k \otimes K_X^*$. Notice that for high enough $k>0$ this bundle is necessarily ample, and $H^0(X,E_k \otimes K_X)=H^0(X,L^k)$. We fix $\o^n$ to be the background metric on $K_X^*$. In fact, with the notation of the proposition we have $\eta_k := \Theta(h^k \otimes \o^n) =k \o + \Ric \o>0$. 

For $u \in \mathcal E^p_{\frac{1}{k}\eta_k} \cap \mathcal E^p_{\o}$, we end up having the following dictionary between our two different notions of Hilbert maps:
$$
{{\hbox{\rm H}}}_k(u)=\Hilb_{E_k \otimes K_X}(ku).
$$
With this identity in hand, we state the following corollary of Proposition \ref{prop: quant_max_princ} that applies for high powers of our ample line bundle $L \to X$, \emph{without} any additional twisting:
\begin{corollary} \label{cor: quant_max_princ} Suppose that $(L,h) \to X$ is ample with $\omega := \Theta(h)>0$, and let $\varepsilon >0$. Suppose that $[0,1] \ni t \to v_t \in \mathcal E_\o^p$ is a weak subgeodesic connecting $v_0,v_1 \in \mathcal E^p_\o$ that satisfies
\begin{equation}\label{eq: geod_strict_pos}
\pi_{S \times X}^* \o + \i\partial_{S \times X} \bar \partial_{S \times X} v \geq \varepsilon \pi_{S \times X}^* \o.
\end{equation}
We introduce $V^k_{v_t} := {{\hbox{\rm H}}}_k(v_t)$, and by 
$[0,1] \ni t \to V^k_t \in \mathcal H_k$
we denote the geodesic connecting $V^k_0,V^k_1 \in \mathcal H_k$, solving \eqref{eq: goed_eq_intr}. Then  there exists $k_0(\varepsilon)>0$ such that for all $k \geq k_0$ the following hold: if $V^k_0 \leq V^k_{v_0}$ and $V^k_1 \leq V^k_{v_1}$ then $V^k_t \leq V^k_{v_t}, \ t \in [0,1].$
\end{corollary}
\begin{proof}
By the discussion preceding the corollary, we have to show existence of $k_0(\varepsilon)>0$ such that $v_t \in \mathcal E^p_{\frac{1}{k}\eta_k}, \ t \in [0,1]$ and that $t \to kv_t$ is a weak $\eta_k$-subgeodesic for all $k \geq k_0$. 

These will follow from the positivity condition \eqref{eq: geod_strict_pos}. Indeed, we immediately see that 
$$\pi^* \eta_k + \i\partial_{S \times X} \bar \partial_{S \times X} kv \geq \varepsilon k \pi^* \o + \pi^* \textup{Ric }\o.$$
Consequently, for $k \geq k_0(\varepsilon)$ the curve $t \to kv_t$ is indeed a weak $\eta_k$-subgeodesic. We fix $t \in [0,1]$. By restricting \eqref{eq: geod_strict_pos} to the appropriate $X$ fiber, we see that 
$$\frac{1}{k}\eta_k + i\ddbar v_t = \o_{v_t} + \frac{1}{k}\textup{Ric} \ \o \geq \frac{1}{k}\textup{Ric} \ \o +\varepsilon \o.$$ 
Consequently, for all $k \geq k_0(\varepsilon)$ we have that $v_t \in \textup{PSH}(X,\frac{1}{k}\eta_k)$. Lemma \ref{lem: E^1_stability} immediately gives that $v_t \in \mathcal E^p_{\frac{1}{k}\eta_k}$, finishing the argument. 
\end{proof}

As in \cite{Bern09}, in case $V_0^k=V^k_{v_0}$ and $V_1^k=V^k_{v_1}$ we would like to compare the derivatives of $t \to V_t^k$ and $t \to V_{v_t}^k$ at the endpoints. However we do not know if the map $t \to {{\hbox{\rm H}}}_k(v_t)$ is differentiable in case $t \to v_t$ is only a weak subgeodesic. To overcome this difficulty, we prove the following partial result that will suffice in later investigations:

\begin{lemma}\label{lem: quant_max_princ_deriv} Suppose that in the setting of Corollary \ref{cor: quant_max_princ} we have that   $t \to v_t$ is increasing, and $V^k_0 = V^k_{v_0}$ along with $V^k_1 = V^k_{v_1}$. Then 
$$-k \int_X  \dot v_1 h_L^k(s,s) e^{-k v_1}\omega^n \leq \dot V_1^k(s,s), \ \ s \in H^0(X,L^k).$$
\end{lemma}

Note that we do not rule out the possibility that the expression on the left hand side might equal $-\infty$ for some $s \in H^0(X,L^k)$. 

\begin{proof}  For $C> 0$ we introduce $[0,1] \ni t \to v^C_t:=\max(v_t, v_1 - C(1-t)) \in \mathcal E^p_{\omega}$. It is easy to see that this curve is still a weak subgeodesic with $v^C_1 = v_1$, such that $$0 \leq \dot v^C_t \leq C, \ t \in [0,1] \ \ \ \ \ \textup{ and } \ \ \ \ \ 
\pi_{S \times X}^* \o + i\partial_{S \times X} \bar \partial_{S \times X} v^C \geq \varepsilon \pi_{S \times X}^* \o.
$$ 
Corollary \ref{cor: quant_max_princ} applies to $t \to v_{t}^C$ and we get that 
\begin{equation}\label{eq: v^C_comp}
V^{k,C}_t \leq {{\hbox{\rm H}}}_k(v^C_t), \ t \in [0,1], \ k \geq k_0,
\end{equation}
where $[0,1] \ni t \to V^{k,C}_t \in \mathcal H_k$ is the geodesic connecting $V^k_{v_{0}^C}$ (which is smaller than $V^k_{v_0}$) and $V^k_{v_1}$. In particular, using \eqref{eq: v^C_comp}, we can compare derivatives at $t=1$ in the following manner: 
\begin{flalign*}
\dot V^{k,C}_1(s,s) & \geq \limsup_{t\to 1^-} \frac{{{\hbox{\rm H}}}_k(v_1)(s,s)-{{\hbox{\rm H}}}_k(v_t^C)(s,s)}{1-t} \\
& \geq \limsup_{t\to 1^-} \frac{\int_X h_L^k(s,s) (e^{-k v_1}-e^{-k v_t^C}) \omega^n}{1-t}\\
& \geq \limsup_{t\to 1^-}  \frac{-k \int_X h_L^k(s,s)  (v_1-v_t^C) e^{-k v_t^C } \omega^n}{1-t},
\end{flalign*}
where in the last line we have used that $e^{x}\geq 1+ x$.
Since $(v_1^C -v_t^C)/(1-t)$ is bounded and $e^{-kv_{t}^C}\leq e^{-k v_1 +kC}$  it follows from the dominated convergence theorem that 
\begin{flalign*}
\dot V^{k,C}_1(s,s) \geq -k \int_X h_L^k(s,s) \dot{v}_1^C e^{-kv_1} \omega^n \geq -k \int_X h_L^k(s,s) \dot{v}_1 e^{-kv_1} \omega^n, \ s \in H^0(X,L^k).
\end{flalign*}
Since solutions to the geodesic equation \eqref{eq: goed_eq_intr} are endpoint stable, we have that $V^{k,C}_t  \to V^{k}_t$ as $C \to +\infty$. Since $V^{k,C}_1 = V^{k}_1$, at $t=1$ we also have convergence of tangent vectors: $\dot V^{k,C}_1 \to \dot V^{k}_1$, hence the conclusion.
\end{proof}

\subsection{Approximation of finite-energy potentials from below}

In this section we show that any potential $u \in \mathcal E^p_\o$ can be approximated from below in a concrete manner by an increasing sequence of strongly $\o$-psh potentials.

\begin{proposition}\label{prop: E_p_approx_below} Suppose $u \in \mathcal E^p_\o$ with $u \leq -1$. Then for any $\delta \geq 1$ we have that $P(\delta u) \in \mathcal E^p_\o$. Moreover $P(\delta u) \nearrow u$ as $\delta \searrow 1$. In particular, $d_p(\frac{1}{\delta}P(\delta u),u) \to 0$ and $d_p(P(\delta u),u) \to 0$ as $\delta \searrow 1$.
\end{proposition}

In fact, the condition $u \leq -1$ can be removed, however we will only use this result in the above form.

\begin{proof}Let $u_j \in \mathcal H_{\omega}$ be a decreasing sequence of negative potentials such that $u_j \searrow u$. Fix $\delta >1$. It is well known that $P(\delta u_j) \in \textup{PSH}(X,\o) \cap C^{1\bar 1}$ (\cite{Berm13}, for a survey see \cite[Theorem A.7]{Da18}) and
\begin{equation}\label{eq: MA_contact_set_est}
\o_{P(\delta u_j)}^n = \mathbbm{1}_{\{P(\delta u_j) =\delta u_j\}} \o^n_{\delta u_j} \leq \mathbbm{1}_{\{P(\delta u_j) =\delta u_j\}} \delta ^n \o^n_{u_j},
\end{equation}
where in the last estimate we have used the multinearity of the complex Monge--Amp\`ere measure and that $\o_{\delta u_j} \leq \delta \o_{u_j}$.

Consequently, we can write that
\begin{flalign*}
\int_X |P(\delta u_j)|^p \o_{P(\delta u_j)}^n &\leq \int_{\{P(\delta u_j) =\delta u_j\}} |P(\delta u_j)|^p \delta^n \omega_{u_j}^n= \delta ^{n+p} \int_{\{P(\delta u_j) =\delta u_j\}} |u_j|^p \o_{u_j}^n\\
& \leq  \delta ^{n+p} \int_X |u_j|^p \o_{u_j}^n.
\end{flalign*}
Now Lemma \ref{lem: uniform_est} and Lemma \ref{lem: fund_ineq} implies that the decreasing limit $P(\delta u) = \lim_{j} P(\delta u_j)$ is an element of $\mathcal E^p_\o$, moreover by Lemma \ref{lem: limit_energy_L^p}:
\begin{equation}\label{eq: P_delta_est}
\int_X |P(\delta u)|^p \o_{P(\delta u)}^n \leq \delta^{n+p} \int_X |u|^p \o_u^n.
\end{equation}
Next we show that $P(\delta u) \nearrow u$ as $\delta \searrow 1$. Since $\{P(\delta u) \}_\delta$ is increasing  and $P(\delta u) \leq u$ it follows that $\lim_{\delta \searrow 1}P(\delta u) =v \in \mathcal E^p_\o$ with $v \leq u$.
 
It follows from \eqref{eq: MA_contact_set_est} that $\o_{P(\delta u_j)}^n \leq \delta^n \o_{u_j}^n$. Taking the limit $j \to \infty$ we obtain that
$$\o_{P(\delta u)}^n \leq \delta^n \o_{u}^n.$$
Taking now the limit $\delta \searrow 1$ we obtain that $\o_v^n \leq \o_u^n$. Since both $u$ and $v$ are elements of $\mathcal E^p_\o$, it follows that this last inequality is in fact an equality, hence $u+c=v$ for some $c \leq 0$ (see \cite{Diw09}). After letting $\delta \searrow 1$  in \eqref{eq: P_delta_est} (and using again Lemma \ref{lem: limit_energy_L^p}) we see that $c=0$, proving that $P(\delta u)$ increases to $u$.

The last statement now readily follows from the fact that $P(\delta u) \leq \frac{1}{\delta}P(\delta u) \leq u$ and \cite[Proposition 4.9]{Da15}.
\end{proof}

\section{The Finsler geometry of Hermitian matrices}

In this section we explore the $L^p$ (or more generally Orlicz) Finsler geometry of $\Pn$.
We will point out that the geodesic curves in the different $L^p$ geometries are the same for any $p \geq 1$, shadowing the analogous results of \cite{Da15} in the infinite-dimensional setting.

$L^p$ norms unfortunately misbehave under differentiation, especially for $p \geq 1$ close to $1$. Because of this we will work with more general smooth Orlicz norms, and finally use approximation  to recover our main results for the (nonsmooth) $L^p$ norms, as in \cite{Da15}.
For an elaborate introduction to Orlicz norms we refer to \cite{RR91}, however the brief self-contained discussion in \cite[Section 1]{Da18} will suffice for our purposes. 

We take $(\chi,\chi^*)$, a complementary pair of Young weights. This means that 
$$\chi:\Bbb R \to \Bbb R^+ \cup \{ \infty \}$$ is convex, even, lower semi-continuous (lsc)  and satisfies the conditions $\chi(0)=0$, $1 \in \partial \chi(1)$. To clarify, $\partial \chi(l) \subset \Bbb R$ is the set of subgradients to $\chi$ at $l \in \Bbb R$. The complement $\chi^*$ is the Legendre transform of $\chi$:
$$
\chi^*(h) = \sup_{l \in \Bbb R} (lh - \chi(l)).
$$
Using convexity of $\chi$ and the above identity, one can verify that $\chi^*$ satisfies the same properties as $\chi$. Lastly, $(\chi,\chi^*)$ satisfies the Young identity and inequality:
\begin{equation}\label{eq: YoungIdIneq}
\chi(a) + \chi^*(\chi'(a))=a\chi'(a), \ \chi(a) + \chi^*(b) \geq ab, \ a,b \in \Bbb R, \ \chi'(a) \in \partial \chi(a).
\end{equation}
Naturally, the typical example to keep in mind is the pair $\chi_p(l)=|l|^p/p$ and $\chi^*_p(l)=|l|^q/q$, where $p,q > 1$ and $1/p + 1/q = 1$. Unfortunately these $L^p$ weights are not smooth for all $p \geq 1$, and we will have to approximate these weights with smooth Orlicz weights $\chi$, for which our initial analysis carries through. For ways on how to do this, we refer to \cite[Proposition 1.7]{Da18}. 

Returning to $\Pn$, we note that this space can be identified with an open subset of $\Bbb R^{n^2}$, hence it has the structure of a trivial $n^2$-dimensional manifold. We now point out how to introduce non-trivial Finsler metrics on $\Pn$, relevant to our work. 

Given $h \in \Pn$, let $\phi \in T_h \Pn=\Herm_n$ (recall \eqref{Hermneq}). We introduce the following $h$-self-adjoint operator $\phi^h \in \textup{GL}(n,\Bbb C)$:
$$\phi(s,s') = h(\phi^h  s, s') = h(s, \phi^h s'), \ \ s,s' \in \Bbb C^n.$$
In matrix notation we simply have that 
$$\phi^h := h^{-1} \cdot \phi.$$
Since $\phi^h$ is $h$-self-adjoint we obtain that it is diagonalizable with only real eigenvalues, i.e., $\phi^h \in \Bbb D^n$, with the notation of Section 2.2. Using the considerations of that same section, for continuous $\chi$  it makes sense to consider $\chi(\phi^h)$, and we introduce the $\chi$-Orlicz Finsler norm  $\| \phi \|_{\chi,h}$ in the following manner:
$$
\| \phi \|_{\chi,h}:=\inf \bigg\{ r > 0 :  \frac{1}{n}{{\tr}}\bigg[\chi\bigg(\frac{\phi^h}{r}\bigg)\bigg]  \leq \chi(1) \bigg\}.
$$
Given our specific setup, it is straightforward to see that ${{\tr}}\big[\chi\big(\frac{\phi^h}{\| \phi \|_{\chi,h}}\big)\big]  =n  \chi(1), $ in case $\chi$ is strictly convex and smooth. 

A few words are in order about why this definition gives a norm on each fiber of $T \Pn$. By convexity of $\chi$, it is well known that $A \to {{\tr}}\big[ \chi(A)\big]$ is convex for self-adjoint matrices $A$ (\cite[Theorem 2.10]{Ca10}). In particular, since  $\phi^h$ is $h$-self-adjoint, it follows that the correspondence
$$T_h \Pn \ni \phi  \to \frac{1}{n}{{\tr}}\big[ \chi(  \phi^h) \big] \in \Bbb R$$ 
is convex. Consequently, $\| \cdot \|_{\chi,h}$ is simply the Minkowski functional of this convex map, hence it is indeed a norm. 

Next, 
we state the matrix version of the Young identity and inequality from \eqref{eq: YoungIdIneq}:

\begin{proposition}\label{prop: Young_eq_ineq_matrix} Let $(\chi,\chi^*)$ be complementary Young weights, with $\chi$ continuous. Suppose $h \in \Pn$ and $u,v \in {\rm Herm}_n$. Then the following hold:\\
(i) $ {{\tr}}\big[\chi(u^h)\big] + {{\tr}}\big[\chi^*(\chi'(u^h))\big]={{\tr}}\big[u^h\chi'(u^h) \big].$\\
(ii) $ {{\tr}}\big[\chi(u^h)\big] + {{\tr}}\big[\chi^*(v^h)\big]\geq{{\tr}}\big[u^h v^h].$
\end{proposition}
\begin{proof} The identity in (i) follows after simply diagonalizing $u^h$ and using \eqref{eq: YoungIdIneq}. 

To argue the estimate of (ii) we need to be slightly more careful. We choose an $h$-orthonormal basis for which $v^h = \textup{diag}(\lambda_1, \ldots,\lambda_n)$. As a result, ${{\tr}}[u^hv^h]=\sum_{j=1}^n \lambda_j u^h_{j\bar j}$, hence we can apply \eqref{eq: YoungIdIneq} to conclude that
$${{\tr}}[u^hv^h] \leq \sum_{j=1}^n \chi(u^h_{j\bar j}) + {{\tr}}\big[ \chi^*(v^h) \big].$$
Peierls' inequality \cite[Theorem 2.9]{Ca10} implies that  $\sum_{j=1}^n \chi(u^h_{j\bar j}) \leq {{\tr}}\big[ \chi(u^h) \big]$, finishing the argument.
\end{proof}

In analogy with \cite[Proposition 3.7]{Da18} (see \cite{Da15} for the original version), we  first establish a formula for the derivative of the Finsler metric:

\begin{proposition}\label{prop: H_norm_deriv} Let $\chi$ be a smooth strictly convex Orlicz weight. Suppose $(0,1) \ni t \to h_t \in \Pn$ and $(0,1) \ni t \to \phi_t \in {\rm Herm}_n$ are smooth curves, with $\phi_t \neq 0, \ t \in (0,1)$. The following formula holds:
$$\frac{d}{dt}\|\phi_t\|_{\chi,h_t} =
\frac{{{\tr}}\Big[\chi'\Big( \frac{\phi^{h_t}_t}{\| \phi_t\|_{\chi,h_t}}\Big)  \frac{d}{dt} \phi_t^{h_t}\Big]}{{{\tr}}\Big[\chi'\Big( \frac{\phi^{h_t}_t}{\| \phi_t\|_{\chi,h_t}}\Big)  \frac{\phi^{h_t}_t}{\| \phi_t\|_{\chi,h_t}}\Big]}.$$
\end{proposition}

The fact that $\chi$ is smooth plays an essential role here. As a courtesy to the reader, we provide the argument that is essentially identical to the proof of \cite[Proposition 3.7]{Da18}.

\begin{proof}   We introduce the smooth function $F: \Bbb R^+ \times (0,1) \to \Bbb R$ given by the formula
$$F(r,t) = {{\tr}}\bigg[\chi\bigg(\frac{\phi_t^{h_t}}{r}\bigg)\bigg] .$$
As $\chi$ is strictly convex and even, along with $\chi(0)=0$, we have that $\chi'(l) >0, \ l >0$ and $\chi'(l) <0, \ l <0$. As $t \to \phi^{h_t}_t$ is non-vanishing, it follows from Proposition \ref{prop: Tr_f_deriv} that
$$\frac{d}{dr} F(r,t)=- \frac{1}{r^2}{{\tr}} \bigg[\chi'\bigg(\frac{\phi_t^{h_t}}{r}\bigg) \cdot \phi_t^{h_t} \bigg] < 0$$
for all $r > 0,\ t \in (0,1)$. Indeed, the matrix whose trace we take above has only non-negative eigenvalues, with at least one non-zero eigenvalue.

Using the above, and the fact that 
$F(\|\phi_t\|_{\chi,h_t},t) = n\chi(1)$,  an application of the implicit function theorem yields that the map $t \to \|\phi_{t}\|_{\chi,h_t}$ is continuously differentiable. Using again Proposition \ref{prop: Tr_f_deriv}, the following formula holds, as proposed:
$$\frac{d}{dt}\|\phi_t\|_{\chi,h_t} =
\frac{{{\tr}}\Big[\chi'\Big( \frac{\phi^{h_t}_t}{\| \phi_t\|_{\chi,h_t}}\Big)  \frac{d}{dt} \phi_t^{h_t}\Big]}{{{\tr}}\Big[\chi'\Big( \frac{\phi^{h_t}_t}{\| \phi_t\|_{\chi,h_t}}\Big)  \frac{\phi^{h_t}_t}{\| \phi_t\|_{\chi,h_t}}\Big]}.$$
\end{proof}

Having an Orlicz--Finsler metric on $T \Pn$, we can define the length 
$$l_\chi(\{h_t\})=l_\chi\big(\{h_t\}_{t\in[0,1]}\big)$$ of any smooth curve $t \to h_t$ in $\Pn$. This leads to the definition of a pseudo-distance on $\Pn$:
\begin{equation}\label{eq: d^chi_H_def}
d_{\chi,\Pn}(v_0,v_1):= \inf \{l_\chi(\{u_t\}) \},
\end{equation}
where the infimum is taken over all smooth curves $[0,1] \ni t\mapsto u_t \in \Pn$ such that $u_0=v_0$ and $u_1=v_1$.

We will show that $d_{\chi,\Pn}$ is indeed a metric, but first we need to 
construct some geodesics of this pseudo-distance.
Any two points of $h_0,h_1 \in \Pn$ can be joined by a \emph{candidate geodesic} $[0,1] \ni t \to h_t \in \Pn$ governed by the following equation:
\begin{equation}\label{eq: H_geod_eq}
\frac{d}{dt} \Big(h^{-1}_t \dot h_t\Big) = \frac{d}{dt} \dot h_t^{h_t} =0.
\end{equation}
This equation is independent of $\chi$, and it is the well known geodesic equation of the $L^2$ geometry. In particular, it is well known (and easy to show) that for any $h_0,h_1$ there exists a unique curve $t \to h_t$ solving the above equation: in an $h_0$-orthonormal basis of $\Bbb C^n$ making $h_1 = \textup{diag}(e^{\lambda_1},\ldots,e^{\lambda_n})$, we simply take $h_t:= \textup{diag}(e^{t\lambda_1},\ldots,e^{t\lambda_n}), \ t \in [0,1]$.

Next we provide the finite-dimensional analog of a result of \cite{Da15}, that in turn builds on calculations of \cite{Ch00} in case of the $L^2$ geometry (see  also \cite[Proposition 3.11]{Da18}):

\begin{proposition}Let $\chi$ be a smooth strictly convex weight. Let $[0,1] \ni s \to h_s \in \Pn$ be a smooth curve and $v \in \Pn \setminus h_{[0,1]}$.  We denote by $[0,1] \times [0,1] \ni (s,t) \to g_{s,t} \in \Pn$ the smooth function for which $t \to g_{s,t}$ is the (candidate) geodesic from \eqref{eq: H_geod_eq} connecting $v$ and $h_s$. Then the following holds:
$$l_\chi(\{g_{0,t}\}) \leq l_\chi(\{h_s\}) + l_\chi(\{g_{1,t}\}).$$
\end{proposition}

\begin{proof}
To avoid cumbersome notation and possible confusion, derivatives in the $t$--direction will be denoted by dots and derivatives in the $s$--direction will be denoted by $\partial_s$.  

By $l_\chi(g_{s,t})$ we denote the $\chi$-length of the curve $t \to g_{s,t}$. Using the previous proposition, we start with the following line of calculation:\begin{flalign}\label{eq: begin_calc}
\partial_sl_\chi(g_{s,t})&= \int_0^1 \partial_s \|\dot g_{s,t}\|_{\chi,g_{s,t}}dt = \int_0^1 \frac{{{\tr}}\Big[\chi'\Big( \frac{\dot g_{s,t}^{g_{s,t}}}{\| \dot g_{s,t}\|_{\chi,g_{s,t}}}\Big)  \partial_s \big(\dot g_{s,t}^{g_{s,t}}\big)\Big]}
{{{\tr}}\Big[\chi'\Big( \frac{\dot g_{s,t}^{g_{s,t}}}{\| \dot g_{s,t}\|_{\chi,g_{s,t}}}\Big)  \frac{\dot g_{s,t}^{g_{s,t}}}{\| \dot g_{s,t}\|_{\chi,g_{s,t}}}\Big]} dt.
\end{flalign}
Next we notice that 
 $$\partial_s \big(\dot g_{s,t}^{g_{s,t}}\big)-\dot{\overline{\big(\partial_s g_{s,t}\big)^{g_{s,t}}}}=-\big(\partial_s g_{s,t} \big)^{g_{s,t}}\cdot \dot g_{s,t}^{g_{s,t}} + \dot g_{s,t}^{g_{s,t}} \cdot \big(\partial_s g_{s,t} \big)^{g_{s,t}}.$$
 Here $\dot{\overline{x}}$ just means the derivative in $t$ (and not complex conjugate). 
By definition (see \eqref{f_matrix_def}) we also have that  
$$\chi'\bigg( \frac{\dot g_{s,t}^{g_{s,t}}}{\| \dot g_{s,t}\|_{\chi,g_{s,t}}}\bigg) \cdot \dot g_{s,t}^{g_{s,t}}=\dot g_{s,t}^{g_{s,t}} \cdot \chi'\bigg( \frac{\dot g_{s,t}^{g_{s,t}}}{\| \dot g_{s,t}\|_{\chi,g_{s,t}}}\bigg). $$ Using the formula ${{\tr}}\big[U\cdot V \big]={{\tr}}\big[V\cdot U \big]$ together with the above observations, we can continue \eqref{eq: begin_calc} in the following manner:
\begin{flalign*}
\partial_sl_\chi(g_{s,t})&= \int_0^1 \frac{{{\tr}}\Big[\chi'\Big( \frac{\dot g_{s,t}^{g_{s,t}}}{\| \dot g_{s,t}\|_{\chi,g_{s,t}}}\Big)  \dot{\overline{\big(\partial_s g_{s,t}\big)^{g_{s,t}}}}\Big]}{{{\tr}}\Big[\chi'\Big( \frac{\dot g_{s,t}^{g_{s,t}}}{\| \dot g_{s,t}\|_{\chi,g_{s,t}}}\Big)  \frac{\dot g_{s,t}^{g_{s,t}}}{\| \dot g_{s,t}\|_{\chi,g_{s,t}}}\Big]} dt
=\int_0^1 \frac{\frac{d}{dt}{{\tr}}\Big[\chi'\Big( \frac{\dot g_{s,t}^{g_{s,t}}}{\| \dot g_{s,t}\|_{\chi,g_{s,t}}}\Big)  \big(\partial_s g_{s,t}\big)^{g_{s,t}}\Big] }
{{{\tr}}\Big[\chi'\Big( \frac{\dot g_{s,t}^{g_{s,t}}}{\| \dot g_{s,t}\|_{\chi,g_{s,t}}}\Big)  \frac{\dot g_{s,t}^{g_{s,t}}}{\| \dot g_{s,t}\|_{\chi,g_{s,t}}}\Big]} dt \\
&=\int_0^1 \frac{d}{dt} \frac{{{\tr}}\Big[\chi'\Big( \frac{\dot g_{s,t}^{g_{s,t}}}{\| \dot g_{s,t}\|_{\chi,g_{s,t}}}\Big)  \big(\partial_s g_{s,t}\big)^{g_{s,t}}\Big] }
{{{\tr}}\Big[\chi'\Big( \frac{\dot g_{s,t}^{g_{s,t}}}{\| \dot g_{s,t}\|_{\chi,g_{s,t}}}\Big)  \frac{\dot g_{s,t}^{g_{s,t}}}{\| \dot g_{s,t}\|_{\chi,g_{s,t}}}\Big]} dt,
\end{flalign*}
where in the last two steps we have used   the geodesic equation \eqref{eq: H_geod_eq} (and Proposition \ref{prop: H_norm_deriv}) to conclude that $ \dot g_{s,t}^{g_{s,t}}$ and $\| \dot g_{s,t}\|_{\chi,g_{s,t}}$ are independent of $t$. We can integrate out this last identity to deduce that
\begin{flalign}\label{eq: lastidd}
\partial_sl_\chi(g_{s,t})& = \frac{{{\tr}}\Big[\chi'\Big( \frac{\dot g_{s,t}^{g_{s,t}}}{\| \dot g_{s,t}\|_{\chi,h_t}}\Big)  \big(\partial_s g_{s,t}\big)^{g_{s,t}}\Big] }
{{{\tr}}\Big[\chi'\Big( \frac{\dot g_{s,t}}{\| \dot g_{s,t}\|_{\chi,g_{s,t}}}\Big)  \frac{\dot g_{s,t}^{g_{s,t}}}{\| \dot g_{s,t}\|_{\chi,g_{s,t}}}\Big]}\Bigg|_{t=1} = \frac{{{\tr}}\Big[\chi'\Big( \frac{\dot g_{s,1}^{h_s}}{\| \dot g_{s,1}\|_{\chi,h_s}}\Big)  \big(\partial_s h_s\big)^{h_s}\Big] }
{{{\tr}}\Big[\chi'\Big( \frac{\dot g_{s,1}^{h_s}}{\| \dot g_{s,1}\|_{\chi,h_s}}\Big)  \frac{\dot g_{s,1}^{h_s}}{\| \dot  g_{s,1}\|_{\chi,h_s}}\Big]} \nonumber\\
&\geq -\big\|\big(\partial_s h_s\big)^{h_s}\big\|_{\chi, h_s},
\end{flalign}
where in the last step we have used the matrix version of the Young identity and inequality (Proposition \ref{prop: Young_eq_ineq_matrix}(i)(ii)) in the following manner:
\begin{flalign*}
\frac{{{\tr}}\Big[\chi'\Big( \frac{\dot g_{s,1}^{h_s}}{\| \dot g_{s,1}\|_{\chi,h_s}}\Big)  \big(\partial_s h_s\big)^{h_s}\Big] }
{\big\|\big(\partial_s h_s\big)^{h_s}\big\|_{\chi, h_s}} 
& \geq- {{\tr}} \bigg[ \chi \bigg( \frac{ \big(\partial_s h_s\big)^{h_s} }{\big\| \big(\partial_s h_s\big)^{h_s}\big\|_{\chi,h_s}}\bigg) + \chi^*\bigg(-\chi' \bigg(\frac{\dot g_{s,1}^{h_s}}{\| \dot g_{s,1}\|_{\chi,h_s}}\bigg)\bigg) \bigg]\\
& = - {{\tr}} \bigg[ \chi \bigg( \frac{ \big(\partial_s h_s\big)^{h_s} }{\big\| \big(\partial_s h_s\big)^{h_s}\big\|_{\chi,h_s}}\bigg) + \chi^*\bigg(\chi' \bigg(\frac{\dot g_{s,1}^{h_s}}{\| \dot g_{s,1}\|_{\chi,h_s}}\bigg)\bigg) \bigg]\\
&=-\bigg(n\chi(1) + {{\tr}}\bigg[\chi^*\bigg(\chi' \bigg(\frac{\dot g_{s,1}^{h_s}}{\| \dot g_{s,1}\|_{\chi,h_s}}\bigg)\bigg)\bigg]\bigg)\\
&=-\bigg({{\tr}}\bigg[\chi \bigg(\frac{\dot g_{s,1}^{h_s}}{\| \dot g_{s,1}\|_{\chi,h_s}}\bigg)\bigg] + {{\tr}}\bigg[\chi^*\bigg(\chi' \bigg(\frac{\dot g_{s,1}^{h_s}}{\| \dot g_{s,1}\|_{\chi,h_s}}\bigg)\bigg)\bigg]\bigg)\\
&=-{{\tr}}\bigg[\chi'\bigg( \frac{\dot g_{s,1}^{h_s}}{\| \dot g_{s,1}\|_{\chi,h_s}}\bigg)  \frac{\dot g_{s,1}^{h_s}}{\| \dot g_{s,1}\|_{\chi,h_s}}\bigg].
\end{flalign*}
Integrating estimate \eqref{eq: lastidd} with respect to $s$ yields the desired inequality.
\end{proof}

Finally we arrive at the main result of this section, that identifies the geodesics of the $\chi$-Finsler geometry and implicity shows that $d_{\chi,\Pn}$, 
defined in \eqref{eq: d^chi_H_def}, is a metric. With the previous proposition in hand we can consider non-smooth Orlicz weights. Though bigger generality is possible, we will restrict our attention to the
following class:

\begin{definition}
\lb{WpDef}
Let $\chi:\Bbb R \to \Bbb R$ be convex, even, lsc, 
such that $\chi(0)=0$, $1 \in \partial \chi(1)$. 
Let $p\ge 1$. We say that $\chi\in\mathcal W^+_p$ 
if
$$l\chi'(l) \leq p \chi(l), \qq \h{for all}\quad l > 0.$$
\end{definition}

Clearly the standard $L^p$ weight is an element of $\mathcal W^+_p$.
In the context of K\"ahler geometry  these weights were introduced by Guedj--Zeriahi \cite{GZ07}, and to learn more about their use we refer to \cite[Chapter 1]{Da18}. 

\begin{theorem}\label{thm: Hn_+_metric_geod} Let $\chi \in \mathcal W^+_p$ for some $p \geq 1$. Suppose $[0,1] \ni t \to h_t \in \Pn$ is a smooth curve and $t \to g_t$ is the candidate geodesic joining $h_0$ and $h_1$ with $h_0 \neq h_1$. Then
$$l_\chi(h_t) \geq l_\chi(g_t).$$
In particular, $d_{\chi,\Pn}(h_0,h_1) = l_\chi(g_t) >0$, hence 
$(\Pn,d_{\chi,\Pn})$ is a bona fide geodesic metric space, whose metric geodesics are governed by the equation \eqref{eq: H_geod_eq}. 
\end{theorem}

For sake of simplicity, in the sequel we will use this result only for the $L^p$ Finsler structure on $\Pn$. Looking at the geodesic equation \eqref{eq: H_geod_eq}, $l_\chi(g_t)$ can be explicitly calculated, yielding the following explicit formula for 
$d_{p,\Pn}$:
\begin{equation}\label{eq: d_p_Q_def}
d_{p,\Pn}(h_0,h_1) = \bigg[ \frac{1}{n}\sum_{j=1}^n |\lambda_j|^p\bigg]^{\frac{1}{p}}, \qq h_0,h_1 \in \mathcal P_n,
\end{equation}
where $e^{\lambda_1},\ldots,e^{\lambda_n}$ are the eigenvalues of $h_1$ with respect to $h_0$.

\begin{proof} We can assume without loss of generality that $g_1 \not \in h_{[0,1)}$. First we assume that $\chi$ is smooth and strictly convex. 

Let $g^j_1 \in \Pn \setminus h_{[0,1]}$ such that the coefficients of $g^j_1$ converge to $g_1$. Then the previous proposition implies that
$$l_\chi(g^j_t) \leq l_\chi(h_t) + l_\chi(\tilde g^j_t),$$
where $t \to g^j_t$ is the (candidate) geodesic joining $g_1^j$ with $g_0$ and $t \to \tilde g^j_t$ is the  (candidate) geodesic joining $g_1^j$ with $g_1$. Letting $j \to \infty$ we obtain the desired estimate:
$$l_\chi(g_t) \leq l_\chi(h_t).$$

For a general Young weight $\chi \in \mathcal W^+_p$, let $\chi_k$ be a sequence of smooth strictly convex weights that converge to $\chi$ uniformly on compact intervals.
Such sequence can always be found by \cite[Proposition 1.7]{Da18}. 
By the above we have that
$$l_{\chi_j}(g_t) \leq l_{\chi_j}(h_t).$$

Then \cite[Proposition 1.6]{Da18} implies that $l_{\chi_j}(h_t) \to l_{\chi}(h_t)$, and $l_{\chi_j}(g_t) \to l_{\chi}(g_t)$ to give the desired estimate, and finish the proof of the theorem.
\end{proof}

\section{Quantization of the $L^p$ Finsler structures}

\subsection{Quantization of points}
Our aim in this subsection is to prove the following result:

\begin{theorem}\label{thm: point_quant} Suppose $u \in \mathcal E^p_\o$. Then $d_p({\FSk} \circ {{\hbox{\rm H}}}_k(u),u) \to 0$.
\end{theorem}

\begin{proof} Since ${\FSk} \circ {{\hbox{\rm H}}}_k(u+c)={\FSk} \circ {{\hbox{\rm H}}}_k(u)+c$, we can assume that $u \leq -1$. Let $u_\delta := \frac{1}{\delta}P(\delta u) \in \mathcal E^p_\o$, and we keep in mind that by Proposition \ref{prop: E_p_approx_below} we can make $d_p(u,u_\delta)$ arbitrarily small by choosing $\delta >1$ close enough to $1$.

Fixing $\delta$ momentarily, choose $C >0$ and $k \geq k_0(\delta)$, as in part (i) of the next proposition. Using the triangle inequality we can start writing the following estimates:
\begin{flalign}\label{eq: d_p_triang_est}
d_p({\FSk} \circ {{\hbox{\rm H}}}_k (u),u) & \leq d_p\Big({\FSk} \circ {{\hbox{\rm H}}}_k (u),u_\delta - \frac{C}{k}\Big)+ d_p\Big(u_\delta - \frac{C}{k},u_\delta\Big) + d_p(u_\delta,u)\nonumber \\
& = d_p\Big({\FSk} \circ {{\hbox{\rm H}}}_k (u),u_\delta - \frac{C}{k}\Big)+ \frac{C}{k} + d_p(u_\delta,u)
\end{flalign}
From \cite[Lemma 5.1]{Da15} we have that
$$d_p\Big({\FSk} \circ {{\hbox{\rm H}}}_k (u),u_\delta - \frac{C}{k}\Big)^p \leq \frac{1}{V} \int_X \Big|{\FSk} \circ {{\hbox{\rm H}}}_k (u)-u_\delta + \frac{C}{k}\Big|^p \o_{u_\delta}^n.$$
Using the dominated convergence theorem and 
Proposition \ref{prop: Hilb_kpoint} below, we conclude that
$$\limsup_k d_p\Big({\FSk} \circ {{\hbox{\rm H}}}_k(u),u_\delta - \frac{C}{k}\Big)^p \leq \frac{1}{V}\int_X |u-u_\delta |^p \o_{u_\delta}^n \leq 2^
{2n+3p+3} d_p(u,u_\delta)^p,$$
where in the last step we used \cite[Theorem 3]{Da15} (see also \cite[Theorem 3.32]{Da18}). 
Putting this back into \eqref{eq: d_p_triang_est} and letting $k \to \infty$, followed by $\delta \searrow 1$, the result follows. 
\end{proof}

As promised in the above argument, we establish the following qualitative estimates and convergence rates for the operator $u \to {\FSk} \circ {{\hbox{\rm H}}}_k(u)$:
\begin{proposition}\label{prop: Hilb_kpoint}Suppose $u \in \mathcal E^p_\o$ and $\delta >1$. Then the following hold:\\
(i)$\frac{1}{\delta}P(\delta u) - \frac{C}{k}\leq {\FSk} \circ {{\hbox{\rm H}}}_k(u)$ for some $C >0$ and $k \geq k_0(\delta,\omega)$.\\
(ii)${\FSk} \circ {{\hbox{\rm H}}}_k(u) \leq \sup_X u + C \frac{\log k}{k}$ for some $C>0$.\\
(iii)${\FSk} \circ {{\hbox{\rm H}}}_k(u)(x) \to u(x)$ for any $x \in X$, away from a pluripolar set.
\end{proposition}
\begin{proof} From Proposition \ref{prop: E_p_approx_below} we know that $u_\delta:=\frac{1}{\delta}P(\delta u) \leq u$, $u_\delta \in \mathcal E^p_\o$, and trivially $\o_{u_\delta} > \frac{\delta - 1}{\delta} \o$. 

For $x \in X$ we can apply Theorem \ref{thm: OhsTak} for $u_\delta$ to conclude that for $k \geq k_0(\delta,\omega)$ there exists $s \in H^0(X,L^k)$ such that $s(x) \neq 0$ and
$$\int_X h_L^k(s,s)e^{-ku} \o^n \leq \int_X h_L^k(s,s)e^{-ku_\delta} \o^n \leq C  h_L^k(s(x),s(x))e^{-ku_\delta(x)}.$$ 
Using the extremal characterization of ${\FSk}$ from \eqref{eq: FSk_extremal_def} we obtain from the above that 
$$u_\delta(x) \leq \frac{1}{k}\log \bigg( \frac{h_L^k(s(x),s(x))}{\int_X h_L^k(s,s)e^{-ku} \o^n}\bigg) + \frac{\log(C)}{k} \leq {\FSk} \circ {{\hbox{\rm H}}}_k(u)(x) + \frac{\log(C)}{k}.$$
This establishes (i). 

Since $u \to {\FSk} \circ {{\hbox{\rm H}}}_k (u)$ is monotone we get that 
$${\FSk} \circ {{\hbox{\rm H}}}_k (u) \leq {\FSk} \circ {{\hbox{\rm H}}}_k (\sup_X u)=\sup_X u+{\FSk} \circ {{\hbox{\rm H}}}_k (0).$$ 
Using the asymptotic expansion of Bouche-Catlin-Tian-Zelditch (ii) follows. 

Lastly, we argue (iii) whose proof is adapted from the justification of \cite[Theorem 7.1]{GZ05}. We fix $s \in H^0(X,L^k)$. Pick $x \in X$ along with a coordinate neighborhood $B(x,2r)$ and a trivialization for $L$ on $B(x,2r)$. 
Using Cauchy's estimate in this local neighborhood, we can start writing:
$$|s(x)|^2 \leq \frac{C}{r^n}\int_{B(x,r)} |s(z)|^2.$$
On $B(x,2r)$ we use the trivialization $h_L = e^{-\varphi}$ for some $\varphi \in C^\infty(B(x,2r))$. Using the above estimate we can continue:
\begin{flalign*}
h_L^k(s(x),s(x))&=|s(x)|^2 e^{-k\varphi(x)} \leq C \frac{ e^{\sup_{B(x,r)}k\varphi}}{r^ne^{k\varphi(x)}}\int_{B(x,r)} h_L^k(s,s) \\
&\leq \frac{C}{r^n} e^{k [\sup_{B(x,r)} u+ \sup_{B(x,r)} \varphi - \varphi(x)]} \int_{X} h_L^k(s,s)e^{-ku}\o^n. 
\end{flalign*}
Fixing $\varepsilon >0$, we can choose the same $r>0$ for all $x \in X$ such that $\sup_{B(x,r)} \varphi - \varphi(x) \leq \varepsilon$. Consequently, the extremal characterization of ${\FSk}$ \eqref{eq: FSk_extremal_def} implies that 
$${\FSk} \circ {{\hbox{\rm H}}}_k (u)(x) \leq \frac{C(r)}{k} + \sup_{B(x,r)} u + {\varepsilon}.$$
Letting $k \to \infty$ we obtain that $\limsup_k {\FSk} \circ {{\hbox{\rm H}}}_k (u)(x) \leq \sup_{B(x,r)} u + \varepsilon$ for all $x \in X$. Using plurisubharmonicity we have $\limsup_{y \to x}u(y)=u(x)$, hence we can let $r \to 0$, and then $\varepsilon \to 0$, to conclude that 
$$\limsup_k {\FSk} \circ {{\hbox{\rm H}}}_k(u)(x) \leq u(x).$$
Now from (i) we have that $P(\delta u) \leq \frac{1}{\delta}P(\delta u) \leq \liminf_k {\FSk} \circ {{\hbox{\rm H}}}_k (u)$. Proposition \ref{prop: E_p_approx_below} implies that $P(\delta u)$ increases a.e. to $u$. Putting everything together, we obtain that  ${\FSk} \circ {{\hbox{\rm H}}}_k (u)(x) \to u(x)$ for any $x \in X$ away from a pluripolar set.
\end{proof}

\subsection{Quantization of geodesics}

Our aim in this subsection is to prove the following result:

\begin{theorem} \label{thm: geod_quant} Suppose $u_0,u_1 \in \mathcal E^p_\o$ and $[0,1] \ni t \to u_t \in \mathcal E^p_\o$ is the $L^p$-finite-energy geodesic connecting $u_0,u_1$. Let $[0,1] \ni t \to U^k_t \in \mathcal H_k$ be the $L^p$-Finsler geodesic joining $U_0^k={{\hbox{\rm H}}}_k(u_0)$ and $U_1^k={{\hbox{\rm H}}}_k(u_1)$. Then 
$$d_p({\FSk}(U^k_t), u_t) \to 0  \ \textup{ as } \ k \to \infty \ \textup{ for any } \ t \in [0,1].$$
\end{theorem}

\begin{proof} As before, we can assume without loss of generality that $u_0,u_1 \leq -1$. Let $u^k_0 := {\FSk} \circ {{\hbox{\rm H}}}_k(u_0) \in \mathcal H_\o$ and $u^k_1 := {\FSk} \circ {{\hbox{\rm H}}}_k(u_1)\in \mathcal H_\o$. 

Based on local properties of psh functions, it is well known that $[0,1] \ni t \to {\FSk}(U_t) \in \mathcal H_\o$ joing $u^k_0$ and $u^k_1$ is a subgeodesic hence the comparison principle gives the estimate:
\begin{flalign}\label{eq: geod_lower_est}
{\FSk}(U^k_t) \leq u^k_t, \ t \in [0,1],
\end{flalign}
where $[0,1] \ni t \to u^k_t \in \mathcal H_\o^{1\bar 1}$ is the weak $C^{1\bar 1}$-geodesic joining $u^k_0$ and $u^k_1$. 

Next we fix $\delta > 1$ and we introduce $u_{\delta,0} := \frac{1}{\delta}P(\delta u_0), \ u_{\delta,1} := \frac{1}{\delta}P(\delta u_1)$. Proposition \ref{prop: E_p_approx_below}
tells us that $u_{\delta,0},u_{\delta,1} \in \mathcal E^p_\o$. With $[0,1] \ni t \to v_{\delta,t} \in \mathcal E^p_\o$ being the finite-energy geodesic connecting $P(\delta u_0),P(\delta u_1) \in \mathcal E^p_\o$, we consider the following subgeodesic connecting $u_{\delta,0}$ and $u_{\delta,1}$:
$$[0,1] \ni t \to u_{\delta,t}:= \frac{1}{\delta} v_{\delta,t} \in \mathcal E^p_\o.$$
It is clear that this last curve is a subgeodesic, with the property that
$$\pi_{S \times X}^* \o + i\ddbar u_{\delta,t} \geq \frac{\delta -1}{\delta}\pi_{S \times X}^* \o.$$
As a result, Corollary \ref{cor: quant_max_princ} is applicable to $t \to u_{\delta,t}$, and we obtain that for $k \geq k_0(\delta)$ 
$$
U_t^k \leq  {{\hbox{\rm H}}}_k(u_{\delta,t}), \ t \in [0,1].
$$
Applying ${\FSk}$ to this inequality we conclude that 
\begin{flalign}\label{eq: geod_upper_est}
{\FSk} \circ {{\hbox{\rm H}}}_k(u_{\delta,t}) \leq {\FSk}(U^k_t), \ t \in [0,1].
\end{flalign}
Using the triangle inequality we can start writing the following:
\begin{flalign*}
d_p(  {\FSk}(U^k_t),u_t) & \leq d_p({\FSk}(U^k_t),{\FSk} \circ {{\hbox{\rm H}}}_k(u_{\delta,t})) + d_p({\FSk} \circ {{\hbox{\rm H}}}_k(u_{\delta,t}),u_t)\\
& \leq d_p(u^k_t, {\FSk} \circ {{\hbox{\rm H}}}_k (u_{\delta, t}))+d_p({\FSk} \circ {{\hbox{\rm H}}}_k(u_{\delta,t}),u_t),
\end{flalign*}
where in the first line we have used the triangle inequality, and in the second line we have used \eqref{eq: geod_lower_est}, \eqref{eq: geod_upper_est} along with \cite[Lemma 4.2]{Da15} (and its short proof). 

By the endpoint stability of finite-energy geodesics \cite[Proposition 4.3]{BDL17} it follows that $d^p(u^k_t,u_t) \to 0$ as $k \to \infty$. On the other hand,  Theorem \ref{thm: point_quant} gives that $d_p({\FSk} \circ {{\hbox{\rm H}}}_k(u_{\delta,t}),u_{\delta,t}) \to 0$. As a result, letting $k \to \infty$ in our last estimate we conclude that  
\begin{flalign*}
\limsup_k d_p( & {\FSk}(U^k_t),u_t) \leq  2d_p(u_{\delta, t},u_t) \leq 2d_p(v_{\delta,t},u_t),
\end{flalign*}
where in the last inequality we used again \cite[Lemma 4.2]{Da15} (and its short proof). Using the endpoint stability of finite-energy geodesics one more time, 
we let $\delta \searrow 1$ and conclude that $\limsup_k d_p({\FSk}(U_t),u_t)=0$.
\end{proof}
\subsection{Quantization of distance}
We recall that in \eqref{eq: d^k_p_def_intr} we introduced the following metric on $\mathcal H_k$ for $p \geq 1$:
\begin{equation}\label{eq: d_{p,k}_def}
d_{p,k} := \frac{1}{k} d_{p,\P_{\textup{d}_k}},
\end{equation}
where $\P_{\textup{d}_k}$ is identified with $\mathcal H_k$. Since the underlying $L^p$ Finsler structure on $\Pn$ is independent of change of basis on $\Bbb C^{\textup{d}_k}$, this definition does indeed make sense. 

In this subsection we prove the following theorem:
\begin{theorem}\label{thm: dist_quant} Let $v_0,v_1 \in \mathcal E^p_\o$. Then $\lim_k d_{p,k}({{\hbox{\rm H}}}_k(v_0), {{\hbox{\rm H}}}_k(v_1)) = d_p(v_0,v_1).$
\end{theorem}
\begin{proof} As usual, without loss of generality we can assume that $v_0,v_1 \leq -1$.
Let $v^l_0,v^l_1 \in \mathcal H_\o$ be decreasing sequences of negative potentials converging to $v_0,v_1$. As a consequence of \cite[Theorem 3.3]{Bern09} it follows that 
$$d_{p,k}({{\hbox{\rm H}}}_k(v^l_0),{{\hbox{\rm H}}}_k(v^l_0)) \to d_p(v^l_0,v^l_1) \ \textup{ as } \ k \to \infty.$$
Using the triangle inequality for $d_{p,k}$ and $d_p$, it suffices to prove that for any $\varepsilon >0$  there exist  $l_0$ such that for all $l \geq l_0$ we have:
\begin{flalign}\label{eq: eps_est}
\limsup_{k \to \infty} d_{p,k}({{\hbox{\rm H}}}_k(v^l_0),{{\hbox{\rm H}}}_k(v_0)) \leq \varepsilon  \ \textup{ and } 
\ \limsup_{k \to \infty} d_{p,k}({{\hbox{\rm H}}}_k(v^l_1),{{\hbox{\rm H}}}_k(v_1)) \leq \varepsilon.
\end{flalign}

It is enough to show this estimate for $v_0$. We will do this using an argument similar in spirit to that of Theorem \ref{thm: geod_quant}. Let us fix $\delta >1$ and $l \in \Bbb N$ momentarily.

Let $[0,1] \ni t \to \psi_t \in \mathcal E^p_\o$ be the (increasing) finite-energy geodesic connecting $P(\delta v_0)$ and $v^l_0$. We also introduce a related increasing subgeodesic, connecting $\frac{1}{\delta}P(\delta v_0) + \frac{\delta -1}{\delta} v^l_0$ and $v^l_0$:
$$[0,1] \ni t \to \phi_{t}:=\frac{1}{\delta}\psi_t  + \frac{\delta -1}{\delta} v^l_0\in \mathcal E^p_\o$$
Since all the potentials involved are negative, we notice the following sequence of inequalities
$$\phi_0 = \frac{1}{\delta}P(\delta v_0) + \frac{\delta -1}{\delta} v^l_0\leq v_0 \leq v^l_0 =\phi_1.$$
Using montonicity of ${{\hbox{\rm H}}}_k$ we automatically get:$${{\hbox{\rm H}}}_k(\phi_1) = {{\hbox{\rm H}}}_k(v^l_0) \leq {{\hbox{\rm H}}}_k(v_0)\leq {{\hbox{\rm H}}}_k(\phi_0).$$
From here, by comparing tangent vectors for geodesics at $t=1$, we deduce that 
\begin{flalign}\label{eq: Hilb_k_estimates}
d_{p,k}({{\hbox{\rm H}}}_k(v_0),{{\hbox{\rm H}}}_k(v^l_0)) = d_{p,k}({{\hbox{\rm H}}}_k(v_0),{{\hbox{\rm H}}}_k(\phi_1)) \leq d_{p,k}({{\hbox{\rm H}}}_k(\phi_0),{{\hbox{\rm H}}}_k(\phi_1)).
\end{flalign}
Before we continue, we notice that
$$\pi^*_{S \times X} \o + \i\ddbar \phi_t \geq \frac{\delta -1}{\delta}\pi^*_{S \times X} \o_{v^l_0}.$$ 
Since $v^l_0 \in \mathcal H_\o$, Corollary \ref{cor: quant_max_princ} is applicable to $t \to \phi_t$, and  we obtain that for all $k \geq k_0(l,\delta)$ the following estimate holds:
$$
V_t \leq {{\hbox{\rm H}}}_k(\phi_t), \ t \in [0,1],
$$
where $t \to V_t$ is the (decreasing) finite-dimensional $L^p$-Finsler geodesic joining $V_0:={{\hbox{\rm H}}}_k(\phi_0)$ and $V_1:={{\hbox{\rm H}}}_k(\phi_1).$ Since $t \to V_t$ and $t \to {{\hbox{\rm H}}}_k(\phi_t)$ share the same endpoints, and $t\to\phi_t$ is increasing, we can use Lemma \ref{lem: quant_max_princ_deriv} to conclude that for any $s \in H^0(X,L^k)$ we have
\begin{equation}\label{eq: V_est}
-\frac{1}{\delta}\int_X \dot \psi_1 h_L^k(s ,s)e^{-k v^l_0} \o^n \leq \frac{1}{k} \frac{d}{dt}\Big|_{t=1} V_t(s,s) \leq 0,
\end{equation}
where the last inequality follows since $t \to V_t$ is decreasing. Note that by Lemma \ref{lem: d_p_dist_smooth_endpoint} below we have that the left hand side is finite. 

Recall that $\textup{d}_k = {\dim H^0(X,L^k)}$, and let $\{ e_j \}_{j=1 \ldots \textup{d}_k}$ be a $V_1$-orthonormal basis of $H^0(X,L^k)$ for which the following Hermitian form is diagonal with eigenvalues  $\{{\lambda_j}\}_{j=1...\textup{d}_k}$:
\begin{flalign*}
(s,s') \to -\frac{1}{\delta}\int_X \dot \psi_1 h_L^k(s ,s')e^{-k v^l_0} \o^n.
\end{flalign*}
Putting together \eqref{eq: d_p_Q_def}, \eqref{eq: d_{p,k}_def} and \eqref{eq: V_est} we can initiate the following sequence of estimates:
\begin{flalign*}
d_{p,k}({{\hbox{\rm H}}}_k(\phi_0),{{\hbox{\rm H}}}_k(\phi_1))^p &= \frac{1}{k^p\textup{d}_k} \textup{tr}\big[\big|V_1^{-1} \dot V_1\big|^p\big] \leq  
\frac{1}{\textup{d}_k}  \sum_{j=1}^{\textup{d}_k}|\lambda_j|^p \\
&= \frac{1}{\textup{d}_k} \sum_{j=1}^{\textup{d}_k}\bigg|\frac{1}{\delta}\int_X \dot \psi_1 h_L^k(e_j,e_j)e^{-k v^l_0} \o^n \bigg|^{p}\\
& \leq \frac{1}{\textup{d}_k\delta^p}  \sum_{j=1}^{\textup{d}_k} \int_X |\dot \psi_1|^p h_L^k(e_j,e_j)e^{-k v^l_0} \o^n \\
& = \frac{1}{\delta^p}  \int_X |\dot \psi_1|^p \bigg[\frac{1}{\textup{d}_k}\sum_{j=1}^{\textup{d}_k}  h_L^k(e_j,e_j)e^{-k v^l_0}\o^n \bigg], 
\end{flalign*}
where in the third line we have used the convexity of $t \to |t|^p$ and  the fact that $h^k_L(e_j,e_j)e^{-kv_0^l}\omega^n$ is a probability measure on $X$.

Looking at the last line, from (the twisted version) of the classical Bergman kernel asymptotic expansion (see \cite[\S4.1]{MM} or \cite[Section 2.5]{BBS08}), and the fact that $\textup{d}_k/k^n \to V$,
it readily follows that the expression in the square brackets converges uniformly to $\frac{1}{V}\o^n_{v^{l}_0}$, as $k \to \infty$. Consequently we get that:$$\limsup_k d_{p,k}({{\hbox{\rm H}}}_k(\phi_0), {{\hbox{\rm H}}}_k(\phi_1))^p \leq \frac{1}{V \delta^p}  \int_X |\dot \psi_1|^p \o^n_{v^{l}_0}.$$
Putting this together with \eqref{eq: Hilb_k_estimates} we deduce that
$$\limsup_k d_{p,k}({{\hbox{\rm H}}}_k(v_0), {{\hbox{\rm H}}}_k(v^l_0))^p \leq \frac{1}{V \delta^p}  \int_X |\dot \psi_1|^p \o^n_{v^{l}_0} = \frac{1}{\delta^p}d_p(P(\delta v_0),v^l_0)^p,$$
where in the last step identity we used Lemma \ref{lem: d_p_dist_smooth_endpoint}, argued below.
Taking the limit $\delta \searrow 1$ we obtain in turn that 
$$\limsup_k d_{p,k}({{\hbox{\rm H}}}_k(v_0), {{\hbox{\rm H}}}_k(v^l_0))^p \leq  d_p(v_0,v^l_0)^p.$$
Since $d_p(v_0,v^l_0)$ becomes arbitrarily small as $l \to \infty$, this fully justifies \eqref{eq: eps_est}.
\end{proof}

\begin{lemma}\label{lem: d_p_dist_smooth_endpoint} Suppose that $u_0 \in \mathcal E^p_\o$ and $u_1 \in \mathcal H_\o$ with $u_0 \leq u_1$. Then we have that
$$d_p(u_0,u_1)^p = \int_X |\dot u_1|^p \o_{u_1}^n,$$
where $[0,1]\ni t \to u_t \in \mathcal E^p_\o$ is the finite-energy geodesic connecting $u_0,u_1$.
\end{lemma}

By adapting the arguments of \cite[Lemma 2.4]{BDL16} to the $L^p$ case, we see that the condition $u_0 \leq u_1$ is in fact superfluous, however the above formulation will suffice for our purposes.

\begin{proof} Without loss of generality we can assume that $u_0 < u_1$. We choose $u^k_0 \in \mathcal H_\o$ such that $u^k_0 \searrow u_0$ and $u_0^k \leq u_1$. Let $t \to u^k_t$ be the $C^{1\bar 1}$-geodesic connecting $u^k_0$ and $u_1$. Then \cite[Theorem 1]{Da15} gives that $d_p(u^k_0,u_1)^p = \int_X |\dot u_1^k|^p \o_{u_1}^n.$
We notice that $0 \leq \dot u_1^k \nearrow \dot u_1$, hence the monotone convergence theorem yields the conclusion.
\end{proof}

\section{Quantization of the Pythagorean formula and
a new Lidskii type inequality}

The goal of this section is to prove Theorem \ref{thm: Quant_Pyt_rooftop_intr}. Before we carry that out, we indicate how Lidskii's inequality (Theorem \ref{thm: Lidski_cor}) will be used in our context:
\begin{theorem} \label{thm: mon_rev_triang_ineq} Let $p \geq 1$.\\
\noindent (i) Let $U,V,W \in \mathcal H_k$ such that $U \leq V \leq W$, then $d_{p,k}(V,W)^p \leq d_{p,k}(U,W)^p - d_{p,k}(U,V)^p.$\\
\noindent (ii) Let $u,v,w \in \mathcal E^p_\o$ such that $u \geq v \geq w$, then $d_p(v,w)^p \leq d_p(u,w)^p - d_p(u,v)^p.$
\end{theorem}

It would be interesting to see if the estimate of (ii) holds in case the K\"ahler metric $\o$ does not necessarily coming from the curvature of a line bundle. Either way, this result represents a significant strengthening of \cite[Lemma 4.2]{Da15}.

\begin{proof} First we point out how (i) implies (ii). Since $u \geq v \geq w$, monotonicity of ${{\hbox{\rm H}}}_k$ implies that
${{\hbox{\rm H}}}_k (u) \leq {{\hbox{\rm H}}}_k (v) \leq {{\hbox{\rm H}}}_k (w)$. Applying (i) to these Hermitian forms we get that
$$d_{p,k}({{\hbox{\rm H}}}_k (v),{{\hbox{\rm H}}}_k (w))^p \leq d_{p,k}({{\hbox{\rm H}}}_k (u),{{\hbox{\rm H}}}_k (w))^p - d_{p,k}({{\hbox{\rm H}}}_k (u),{{\hbox{\rm H}}}_k (v))^p.$$
An application of Theorem \ref{thm: dist_quant} now yields (ii).

Now we argue (i). After choosing a $U$-orthonormal basis of $H^0(X,L^k)$, we can assume that $U$ is the identity matrix. Comparing with \eqref{eq: d_p_Q_def}, the desired estimate is equivalent to the following one:
$$
{{\tr}} \big[(\log V^{-1}W)^p\big]  + {{\tr}}\big[ (\log V)^p \big]\leq {{\tr}} \big[(\log W)^p\big].$$
This is exactly what was proved in Theorem \ref{thm: Lidski_cor}.
\end{proof}

The next auxilliary result states that rooftop envelopes are stable under ``variation of the roof'' in $\mathcal E^p_\o:$

\begin{lemma}\label{lem: P_stab} Suppose $u^j_0,u_0,u^j_1,u_1 \in \mathcal E^p_\o$ such that $d_p(u^j_0,u_0) \to 0$ and $d_p(u^j_1,u_1) \to 0$.  Then $d_p(P(u_0^j,u_1^j),P(u_0,u_1)) \to 0$.
\end{lemma}
\begin{proof} We start by applying the triangle inequality:
\begin{equation}\label{eq: triang_ineq}
d_p(P(u_0^j,u_1^j),P(u_0,u_1)) \leq d_p(P(u_0^j,u_1^j),P(u^j_0,u_1)) + d_p(P(u_0^j,u_1),P(u_0,u_1)) 
\end{equation}
Given $u,v,w \in \mathcal{E}^p_{\omega}$, we recall the projection type inequality of \cite[Proposition 4.12]{Da17}:
$$
d_p(P(u,v),P(u,w)) \leq d_p(v,w). 
$$
Using this twice in  \eqref{eq: triang_ineq} we arrive at
$d_p(P(u_0^j,u_1^j),P(u_0,u_1)) \leq d_p(u_1^j,u_1) + d_p(u_0^j,u_0).$
Since $d_p(u_1^j,u_1),d_p(u_0^j,u_0) \to 0$, the result follows.
\end{proof}

Next we prove Theorem \ref{thm: Quant_Pyt_rooftop_intr}(ii):

\begin{theorem}\label{eq: Pyt_limit} Suppose that $u_0,u_1 \in \mathcal E^p_\o$. Then the following hold:
$$\lim_k d_{p,k}\big({{\hbox{\rm H}}}_k(u_0), {{P}}_k({{\hbox{\rm H}}}_k(u_0),{{\hbox{\rm H}}}_k(u_1)) \big)=d_p(u_0,P(u_0,u_1)),$$
$$\lim_k d_{p,k}\big({{\hbox{\rm H}}}_k(u_1), {{P}}_k({{\hbox{\rm H}}}_k(u_0),{{\hbox{\rm H}}}_k(u_1)) \big)=d_p(u_1,P(u_0,u_1)).$$
\end{theorem}

\begin{proof}
We prove both of the identities simultaneously. The proof is carried out in multiple steps. Let us set the stage.

By Theorem \ref{thm: dist_quant} the sequence $d_{p,k}({{\hbox{\rm H}}}_k(u_0),{{\hbox{\rm H}}}_k(u_1))$ converges to $d_p(u_0,u_1)$, hence it is bounded. Consequently,  the quantum Pythagorean formula \eqref{eq: Pyth_formula_quant_intr} implies that  the sequences 
$$l_0^k:=d_{p,k}({{\hbox{\rm H}}}_k(u_0),{{P}}_k({{\hbox{\rm H}}}_k(u_0),{{\hbox{\rm H}}}_k(u_1)))
 \ \ \textup{ and } \ \ \ l^k_1:=d_{p,k}({{\hbox{\rm H}}}_k(u_1),{{P}}_k({{\hbox{\rm H}}}_k(u_0),{{\hbox{\rm H}}}_k(u_1)))$$ 
are bounded too. To finish the proof, it is enough to show that the cluster set of the first sequence contains only $d_p(u_0,P(u_0,u_1))$, and similarly, the cluster set of the second sequence only contains $d_p(u_1,P(u_0,u_1))$. For this, we choose a convergent subsequence  $l^{k_j}_0$.
By the quantum Pythagorean formula  \eqref{eq: Pyth_formula_quant_intr}, $l^{k_j}_1$ has to converge as well, and their limit $l_0$ and $l_1$ satisfies $d_p(u_0,u_1)^p = l_0^p + l_1^p.$
Comparing this equation with the Pythagorean formula \eqref{eq: Pythagorean_intr}, we are done if we can show that $l_0 \geq d_p(u_0,P(u_0,u_1))$ and $l_1 \geq d_p(u_1,P(u_0,u_1))$. The first estimate immediately follows after putting together the estimates of the next two claims. The second one follows after reversing the roles of $u_0,u_1$. \vspace{0.1cm}

\noindent \textbf{Claim 1.} For any sequence $k_j \to \infty$ we have
\begin{flalign*}\limsup_{k_j} d_p(\FS_{k_j}  \circ {{\hbox{\rm H}}}_{k_j}(u_0), & \FS_{k_j}( P_{k_j}({{\hbox{\rm H}}}_{k_j}(u_0),{{\hbox{\rm H}}}_{k_j}(u_1)))) \leq  \nonumber \\
& \leq  \limsup_{k_j} d_{p,k_j}({{\hbox{\rm H}}}_{k_j}(u_0),P_{k_j}({{\hbox{\rm H}}}_{k_j}(u_0),{{\hbox{\rm H}}}_{k_j}(u_1))).
\end{flalign*}
To prove this, let $v \in \mathcal H_\o$ and $\varepsilon >0$ such that $u_0 \leq v$ and $d_p(u_0,v) \leq \varepsilon$. From Theorem \ref{thm: dist_quant} it follows that $\lim_k d_{p,k}({{\hbox{\rm H}}}_k(u_0),{{\hbox{\rm H}}}_k(v)) \leq \varepsilon$. The triangle inequality then implies that:
\begin{flalign}\label{eq: limsup_P}
\limsup_{k_j} d_{p,k_j}({{\hbox{\rm H}}}_{k_j}(v)&,P_{k_j}({{\hbox{\rm H}}}_{k_j}(u_0),{{\hbox{\rm H}}}_{k_j}(u_1))) \leq \nonumber \\
&\leq\limsup_{k_j} d_{p,k_j}({{\hbox{\rm H}}}_{k_j}(u_0),P_{k_j}({{\hbox{\rm H}}}_{k_j}(u_0),{{\hbox{\rm H}}}_{k_j}(u_1))) +\varepsilon.
\end{flalign}
Suppose that $[0,1] \ni t \to v^k_t \in \mathcal E^p_\o$ is the $C^{1,\bar{1}}$ weak geodesic connecting ${\FSk} \circ {{\hbox{\rm H}}}_k(v)$ and ${\FSk} \circ {{P}}_k({{\hbox{\rm H}}}_k(u_0),{{\hbox{\rm H}}}_k(u_1))$.
Also, let $[0,1] \ni t \to V^k_t \in \mathcal H_k$ be the geodesic connecting ${{\hbox{\rm H}}}_k(v)$ and ${{P}}_k({{\hbox{\rm H}}}_k(u_0),{{\hbox{\rm H}}}_k(u_1))$. Note that $V^k_t$ is of the following form:
$$V^k_t(s,s) = \sum_{j=1}^{d_k}|s_j|^2\int_X h^k_L(e_j,e_j)e^{-kv + \lambda_j^kt} \o^n, \ \ \ s = \sum_{j=1}^{d_k} s_j e_j \in H^0(X,L^k),$$
where $\{e_j\}_j$ is a $V^k_0$-orthonormal basis of $H^0(X,L^k)$, diagonalizing $V^k_1$ with eigenvalues $\{e^{\lambda_j^k}\}_j$.
From here, an elementary calculation yields that
$${\FSk} \big(V^k_t\big) = \frac{1}{k} \log \bigg(\sum_{j=1}^{\textup{d}_k}  h^k_L(e_j,e_j)e^{ - \lambda_j^kt} \bigg).$$
To continue, we note that the curve $t \to {\FSk} \big(V^k_t\big) $ is a subgeodesic connecting $v^k_0$ and $v^k_1$, hence ${\FSk} \big(V^k_t\big) \leq v_t^k$. This yields $\frac{d}{dt}\big|_{t=0}{\FSk} \big(V^k_t\big) \leq \dot v^k_0 \leq 0$, allowing to compare the  $L^p$ length of the tangent vectors at $t=0$:
\begin{flalign*}
d_p({\FSk}  \circ &{{\hbox{\rm H}}}_k(v),  {\FSk}( {{P}}_k({{\hbox{\rm H}}}_k(u_0),{{\hbox{\rm H}}}_k(u_1))))^p=\frac{1}{V}\int_X |\dot v_0^k|^p \o_{v_0^k}^n \leq \frac{1}{V} \int_X \Big|\frac{d}{dt}\big|_{t=0}{\FSk} \big(V^k_t\big)\Big|^p \o_{v_0^k}^n\\ & = \frac{1}{V k^p}\int_X \bigg| \frac{-\sum_l \lambda_l^k h^k(e_l,e_l)}{\sum_l h^k(e_l,e_l)}\bigg|^p \o_{v_0^k}^n  \leq \frac{1}{V k^p}\int_X  \frac{\sum_l |\lambda_l^k|^p h^k(e_l,e_l)e^{-kv}}{\sum_l h^k(e_l,e_l)e^{-kv}} \o_{v_0^k}^n,
\end{flalign*}
where we used the convexity of the map $t \to |t|^p$. 
Using the Bouche-Catlin-Lu-Tian-Zelditch expansion and the fact that $ \textup{d}_k/k^n \to V$ (the same way as in the proof of Theorem \ref{thm: dist_quant}) we get that 
$$\frac{1}{\textup{d}_k}\sum_{l=1}^{\textup{d}_k}  h^k(e_l,e_l)e^{-k v}\o^n$$ converges uniformly to $\frac{1}{V}\o_v^n$. For similar reasons $\o^n_{v_0^k}$ converges uniformly to $\o_v^n$ as well. Putting these facts together, we arrive at:
\begin{flalign}\label{eq: dvdvest}
\limsup_{k_j} d_p(&\FS_{k_j}  \circ {{\hbox{\rm H}}}_{k_j}(v),  \FS_{k_j}( P_{k_j}({{\hbox{\rm H}}}_{k_j}(u_0),{{\hbox{\rm H}}}_{k_j}(u_1))))^p \leq  \\
& \leq \limsup_{k_j} \frac{1}{k_j^p}\frac{\sum_l |\lambda_l^{{k_j}}|^p}{\textup{d}_{k_j}}=  \limsup_{k_j} d_{p,k_j}({{\hbox{\rm H}}}_{k_j}(v),P_{k_j}({{\hbox{\rm H}}}_{k_j}(u_0),{{\hbox{\rm H}}}_{k_j}(u_1)))^p. \nonumber
\end{flalign}
Now we use the monotonicity of ${\FSk} \circ {{\hbox{\rm H}}}_k$ to conclude that  
$$ {\FSk} \circ {{\hbox{\rm H}}}_k (v) \geq {\FSk} \circ {{\hbox{\rm H}}}_k (u_0) \geq {\FSk} ({{P}}_k({{\hbox{\rm H}}}_k(u_0),{{\hbox{\rm H}}}_k(u_1))).$$ 
Using this, via Theorem \ref{thm: mon_rev_triang_ineq} we conclude that: 
\begin{flalign*}
\limsup_{k_j} d_p(\FS_{k_j} & \circ {{\hbox{\rm H}}}_{k_j}(u_0),  \FS_{k_j} ( {{P}}_{k_j}({{\hbox{\rm H}}}_{k_j}(u_0),{{\hbox{\rm H}}}_{k_j}(u_1)))) \leq \\
&\leq \limsup_{k_j} d_{p}(\FS_{k_j} \circ \ {\hbox{\rm H}}_{k_j}(v), \FS_{k_j} ( {{P}}_{k_j}({{\hbox{\rm H}}}_{k_j}(u_0),{{\hbox{\rm H}}}_{k_j}(u_1))))\\
&\leq  \limsup_{k_j} d_{p,{k_j}}({{\hbox{\rm H}}}_{k_j}(v),{{P}}_{k_j}({{\hbox{\rm H}}}_{k_j}(u_0),{{\hbox{\rm H}}}_{k_j}(u_1)))\\
&\leq  \limsup_{k_j} d_{p,{k_j}}({{\hbox{\rm H}}}_{k_j}(u_0),{{P}}_{k_j}({{\hbox{\rm H}}}_{k_j}(u_0),{{\hbox{\rm H}}}_{k_j}(u_1))) + \varepsilon,
\end{flalign*}
where in the third line we have used \eqref{eq: dvdvest}, and in the last line we have used \eqref{eq: limsup_P}. 
Letting $\varepsilon \to 0$, the claim follows.\vspace{0.1cm}

\noindent \textbf{Claim 2.} 
$d_p(u_0, P(u_0,u_1)) \leq \liminf_k  d_p({\FSk}\circ {{\hbox{\rm H}}}_k(u_0), {\FSk}({{P}}_k({{\hbox{\rm H}}}_k(u_0),{{\hbox{\rm H}}}_k(u_1)))).\vspace{0.1cm}$

Since ${{\hbox{\rm H}}}_k(u_0),{{\hbox{\rm H}}}_k(u_1) \leq {{P}}_k({{\hbox{\rm H}}}_k(u_0),{{\hbox{\rm H}}}_k(u_1))$, monotonicity of ${\FSk}$ and the definition of $P$ implies that 
$${\FSk}({{P}}_k({{\hbox{\rm H}}}_k(u_0),{{\hbox{\rm H}}}_k(u_1))) \leq P({\FSk}\circ {{\hbox{\rm H}}}_k(u_0),{\FSk} \circ {{\hbox{\rm H}}}_k(u_1)) \leq {\FSk}\circ {{\hbox{\rm H}}}_k(u_0).$$
Consequently, using Theorem \ref{thm: mon_rev_triang_ineq} we conclude that
\begin{flalign*}
d_p({\FSk}\circ {{\hbox{\rm H}}}_k(u_0), & P({\FSk}\circ {{\hbox{\rm H}}}_k(u_0),{\FSk} \circ {{\hbox{\rm H}}}_k(u_1))) \leq \\
& \leq d_p({\FSk}\circ {{\hbox{\rm H}}}_k(u_0), {\FSk}({{P}}_k({{\hbox{\rm H}}}_k(u_0),{{\hbox{\rm H}}}_k(u_1)))). \nonumber
\end{flalign*}
Using Theorem \ref{thm: dist_quant} multiple times together with Lemma \ref{lem: P_stab}, we get that the expression on the left hand side converges to $d_p(u_0, P(u_0,u_1))$, proving the claim.
\end{proof}

Lastly, we prove Theorem \ref{thm: Quant_Pyt_rooftop_intr}(i):
 
\begin{theorem}
Suppose $u_0,u_1 \in \mathcal E^p_\o$. Then $d_p(P(u_0,u_1), {\FSk}({{P}}_k({{\hbox{\rm H}}}_k(u_0), {{\hbox{\rm H}}}_k(u_1))))  
\to 0.$
\end{theorem}

\begin{proof} As it turns out, we already carried out most of the hard work in the proof of the previous theorem. Recall that Theorem \ref{thm: point_quant} applied multiple times, together with Lemma \ref{lem: P_stab} implies that
\begin{flalign} \label{eq: first_lim}
\lim_k d_p({\FSk}\circ {{\hbox{\rm H}}}_k(u_0), P({\FSk}\circ {{\hbox{\rm H}}}_k(u_0),{\FSk} \circ {{\hbox{\rm H}}}_k(u_1))) = d_p(u_0, P(u_0,u_1)).
\end{flalign}
Furthermore Theorem \ref{eq: Pyt_limit} together with Claim 1 from its proof imply that
$$\limsup_k d_p({\FSk}\circ {{\hbox{\rm H}}}_k(u_0), {\FSk}({{P}}_k({{\hbox{\rm H}}}_k(u_0),{{\hbox{\rm H}}}_k(u_1)))) \leq d_p(u_0, P(u_0,u_1)).$$
This together with Claim 2 from the same argument now readily gives that 
\begin{flalign}\label{eq: sec_lim}
\lim_k d_p({\FSk}\circ {{\hbox{\rm H}}}_k(u_0), {\FSk}({{P}}_k({{\hbox{\rm H}}}_k(u_0),{{\hbox{\rm H}}}_k(u_1)))) = d_p(u_0, P(u_0,u_1)).
\end{flalign}

Since ${\FSk}({{P}}_k({{\hbox{\rm H}}}_k(u_0),{{\hbox{\rm H}}}_k(u_1))) \leq P({\FSk} \circ {{\hbox{\rm H}}}_k(u_0),{\FSk} \circ {{\hbox{\rm H}}}_k(u_1)) \leq {\FSk}\circ {{\hbox{\rm H}}}_k(u_0)$, Theorem \ref{thm: mon_rev_triang_ineq} is applicable, and \eqref{eq: first_lim} together with \eqref{eq: sec_lim}  imply that 
\begin{equation}\label{eq: last_est}
d_p({\FSk}({{P}}_k({{\hbox{\rm H}}}_k(u_0),{{\hbox{\rm H}}}_k(u_1)),P({\FSk} \circ {{\hbox{\rm H}}}_k(u_0),{\FSk} \circ {{\hbox{\rm H}}}_k(u_1))) \to 0.
\end{equation}
From Lemma \ref{lem: P_stab} and Theorem \ref{thm: point_quant}, $d_p(P(u_0,u_1),P({\FSk} \circ {{\hbox{\rm H}}}_k(u_0),{\FSk} \circ {{\hbox{\rm H}}}_k(u_1))) \to 0.$ This together with \eqref{eq: last_est} and the triangle inequality implies that
$$d_p({\FSk}({{P}}_k({{\hbox{\rm H}}}_k(u_0),{{\hbox{\rm H}}}_k(u_1))),P(u_0,u_1)) \to 0,$$ what we wanted to prove.
\end{proof}

\medskip

\def\bi{\bibitem}

\footnotesize
\let\OLDthebibliography\thebibliography 
\renewcommand\thebibliography[1]{
  \OLDthebibliography{#1}
  \setlength{\parskip}{1pt}
  \setlength{\itemsep}{1pt}
}

\bigskip
\normalsize
\noindent{\sc University of Maryland}\\
\vspace{0.2cm}\noindent {\tt tdarvas@math.umd.edu, yanir@umd.edu} \\
\noindent {\sc Universit\'e Paris-Sud}\\
\vspace{0.2cm}\noindent{\tt hoang-chinh.lu@u-psud.fr}\\
\end{document}